\documentclass[12pt]{amsart}

\usepackage{times}
\usepackage{setspace}
\usepackage{lipsum}
\usepackage{amscd,amssymb,amsmath,amsthm}
\usepackage{mathtools}
\usepackage[english]{babel}
\usepackage[T1]{fontenc}
\usepackage[arrow,matrix]{xy}
\usepackage{graphicx,subfigure,tikz,tikz-3dplot}
\usetikzlibrary{positioning,matrix,arrows,calc}
\usepackage{cite}
\usepackage{dsfont}
\usepackage{enumitem}
\usepackage{bbm}
\usepackage{hyperref}
\hypersetup{
	colorlinks=true,
	linkcolor=blue}
\usepackage{array}

\oddsidemargin=-0.2in \evensidemargin=-0.2in \theoremstyle{plain}
\newtheorem{theorem}{Theorem}[section]
\newtheorem{lemma}[theorem]{Lemma}
\newtheorem{proposition}[theorem]{Proposition}
\newtheorem{corollary}[theorem]{Corollary}

\newtheorem*{theorem A}{Theorem A}
\newtheorem*{theorem C}{Theorem C}
\newtheorem*{theorem B}{Theorem B}
\newtheorem*{theorem D}{Theorem D}

\theoremstyle{definition}
\newtheorem{remark}[theorem]{Remark}
\newtheorem{definition}[theorem]{Definition}

\numberwithin{equation}{section} 

\newcounter{hypocounter}
\renewcommand\thehypocounter{(H\arabic{hypocounter})}
\begin{document}
\title[locally conformal expanding Actions ]{ locally conformal expanding Actions: Markov Partition and Thermodynamic of the induced skew product}
\author[Ehsani]{Azam Ehsani}
\address{\centerline{Department of Mathematics, Ferdowsi University of Mashhad, }
\centerline{Mashhad, Iran.}  }
\email{aza\_ehsani@yahoo.com}
\author[Fakhari]{Abbas Fakhari}
\address{\centerline{Department of Mathematics, Shahid
Beheshti University,}
\centerline{Tehran, Iran}  }
\email{a\_fakhari@sbu.ac.ir}
\author[Ghane]{Fatemeh Helen Ghane$^{*}$}
\address{\centerline{Department of Mathematics, Ferdowsi University of Mashhad, }
\centerline{Mashhad, Iran.}  }
\email{ghane@math.um.ac.ir}
\thanks{$^*$Corresponding author}
\author[Nazarian]{Javad Nazarian Sarkooh}
\address{\centerline{Department of Mathematics, Ferdowsi University of Mashhad, }
\centerline{Mashhad, Iran.}  }
\email{javad.nazariansarkooh@gmail.com}
\begin{abstract}
For topologically mixing locally conformal semigroup actions generated by a finite collection of $C^{1+\alpha}$ conformal local diffeomorphisms, we provide a countable Markov partition satisfying the finite images and the finite
cycle properties. We show that they admit inducing schemes and describe the tower constructions
associated with them. An important feature of these towers is that their induced maps are equivalent to a subshift of
countable type. Through the investigating the ergodic properties of induced map,
we prove the existence of  liftable measures and establish a thermodynamic formalism of the induced skew product with respect to them.
\end{abstract}
\maketitle
\begingroup
\hypersetup{hidelinks}
\setcounter{tocdepth}{3}
\tableofcontents
\endgroup
\section{Introduction}
The inducing schemes and the corresponding tower constructions are classical objects in ergodic
theory and were considered in works of Kakutani \cite{K}, Rokhlin \cite{Ro}, Pesin and Senti \cite{PS} and others.
In this article, we mainly describe an abstract inducing scheme for a locally conformal expanding semigroup action generated by a finite collection of $C^{1+\alpha}$ (local) diffeomorphisms on a compact Riemannian manifold which parameterized by random
walks. This scheme
provides a symbolic representation over $(W,F,\tau)$, where $F$ is the induced map acting on the inducing
domain $W$ and $\tau$ is the inducing time. This created symbolic dynamic allowing us to apply standard techniques from symbolic dynamics as well. This means that these actions can be partially described, in a symbolic way, as a subshift on a countable set of states. In particular, some basic facts about
the original action can be achieved through the identifying the induced  symbolic maps.
This is a standard way of understanding partial facts about the initial dynamic. For instance, in the uniformly hyperbolic case, Bowen in \cite{B3} used a Markov partition as a tool to present a fairly complete description of ergodic properties of the dynamic.In \cite{BS}, Bowen and Series described Markov partitions related to the
action of certain groups of hyperbolic isometries on the boundary space.
After Bowen, the scenario has been developed to the context of non-uniformly hyperbolic systems (see for instance \cite{BK, MS, O, PS}).
In the non-uniformly hyperbolic setting, we cannot expect to have finite Markov partitions.
However, in many cases it is possible to use the symbolic approach by finding
a countable Markov partition,  or a more general inducing scheme \cite{PS,Y}. This strategy applied by Hofbauer in \cite{Ho}. He used a
countable-state Markov model to study equilibrium states for piecewise monotonic interval maps.
After that Buzzi \cite{Bu} generalized the connected Markov diagrams for nonlinear multi-dimensional dynamical systems by using the Hofbauer's method.
The connected Markov diagrams composed of pairwise disjoint connected subsets which satisfy the Markov property and they
can be used to construct ergodic invariant measures with maximal entropy \cite{Bu}
and equilibrium states for piecewise expanding maps \cite{BS1}.
Yuri introduced in \cite{Y} countable to one Markov systems with finite
range structure which gives a nice countable states symbolic dynamics.
For countable to one transitive Markov systems, she established thermodynamic
formalism for non-H$\ddot{\text{o}}$lder potentials in non-hyperbolic situations. Inducing schemes also use the tower
approach to model the system by a countable-state Markov shift \cite{PS, PSZ2}.
Every inducing scheme generates a symbolic representation by a tower
which is well adapted to construct absolutely continuous measures and equilibrium measures for a certain
class of potential functions using the formalism of countable
state Markov shifts. The projection of these measures from the tower
are natural candidates for equilibrium measures for the original
system.
Apart from the advantage of returning to the initial dynamic,  studying the formalism of the induced map, as a typical example of countable Markov map, can also be interesting on its own.

Thermodynamic formalism provides some methods for constructing invariant measures with
strong statistical properties such as Gibbs measures, physical measures or SRB measures.
It was brought from statistical mechanics to dynamical systems
by the classical works of Ruelle \cite{R1, R2},
Sinai \cite{S2, S3}, and Bowen \cite{B1, B2, B3}.
Sarig in series of papers, \cite{Sa1,Sa2,Sa3,Sa4,Sa5}, has made many fundamental contributions to thermodynamics of countable Markov chains. He adapted transience, null recurrence, and positive
recurrence for non-zero potential functions. On the other hand, for continuous $\mathbb{Z}^d$ actions on compact spaces, Ruelle \cite{R2} obtained a variational principle for the topological
pressure to constructing equilibrium states as the class of pressure maximizing invariant probability measures.
One of the main challenges in this literatures is the extension of the variational principle for more general semigroup actions face some nontrivial challenges.
Carvalho, Rodrigues and Varandas \cite{RV} established a variational principle for finitely generated semigroup actions with finite topological entropy.
The topological aspects of the thermodynamical formalism
for finitely generated semigroup actions were described in \cite{RV}.
Thermodynamic formalism of semigroup actions were also investigated by \cite{Bi1, Bi2, Bi3, But, MW}.

One of the aim in this paper is to develop a thermodynamic formalism
for a certain class of locally conformal expanding semigroups.
However, a unified approach to the thermodynamic formalism for semigroup actions in the
absence of probability measures invariant under all elements of the semigroup. 
Meanwhile, we show how the choice of the random walk in the semigroup settles the ergodic properties
of the action.
We construct a special countable Markov partition for such actions which is
achieved primarily through a finite cover and satisfies the finite cycle property.
We generalize the notion of an inducing scheme to this kind of semigroup actions for generating a symbolic representation.
Then, we use the symbolic approach to study equilibrium states of an appropriate class of real-valued potential functions.
We describe the thermodynamic formalism of the induced map through studying its Gibbs and equilibrium states. For maps with inducing schemes for some classes of potential functions, Pesin and Senti introduced some methods in \cite{PS} to prove the existence of unique equilibrium
measures within the class of all liftable measures. In \cite{PSZ2}, Pesin, Senti and Zhang solved the liftability problem for towers of hyperbolic type
and developed a method to study the liftability problem for this kind of inducing schemes.
Here, using the tower construction of locally conformal expanding semigroup actions, we define an inducing scheme for skew products induced by the semigroup actions. We introduce liftable measures and
describe those measures that give positive weight to the base of the tower and can be lifted. Then, using
 the symbolic representation by the tower, we establish a thermodynamic formalism for liftable measures.
\section{Setting and main results}\label{section2}
In this section, we describe precisely the notions in this paper and our main results.
Throughout the paper $M$ is a compact connected smooth Riemannian manifold and $\lambda$ denotes the normalized
Lebesgue measure.
\subsection{Markov partition for Locally conformal expanding semigroup actions}
Given a finite set of continuous maps $f_i : M \to M$, $i=1, 2, \ldots , k$,
$k > 1$, and the finitely generated semigroup $\Gamma$ (under function composition) with the finite set of
generators $\Pi_1 = \{id, f_1, \ldots, f_k\}$, we will write
\begin{equation*}
  \Gamma:=\bigcup_{n \in \mathbb{N}}\Pi_n,
\end{equation*}
where $\Pi_0=\{id\}$ and $f \in \Pi_n$ if and only if $f=f_{i_n}\circ \cdots \circ f_{i_1}$, with $f_{i_j}\in \Pi_1$.
We will assume that the generator set $\Pi_1$ is minimal, meaning
that no function $f_j$ , for $j = 1, \ldots, k$, can be expressed as a composition from the remaining
generators.
Observe that each element $f$ of $\Pi_n$ may be seen as a word that originates
from the concatenation of $n$ elements in $\Pi^\ast=\{f_i: i=1, 2, \ldots , k\}$.
If $f$ is the concatenation of $n$ elements of $\Pi^\ast$, we will write $|f|:=n$.


Given an open set $U \subset M$, a differentiable local diffeomorphism $f$ is a \emph{$U$-conformal} map if there exists a function
$\alpha :U \to \mathbb{R}$ such that for all $x \in U$ we have
that $Df(x) = \alpha(x) \textnormal{Isom}(x)$, where
$\textnormal{Isom}(x)$ denotes an isometry of $T_xU$. In this case,
$\alpha(x) = \|Df(x)\| $. 
Clearly, any local diffeomorphism in dimension one is a conformal map on the whole space.
\begin{definition}[Locally conformal semigroup]\label{expand}
We say that the semigroup $\Gamma$ is a locally conformal
expanding semigroup if it satisfies the following conditions:
\begin{enumerate}
\item [(\textbf{C1})]For the generators $ f_{1},\ldots,f_{k}$, there exist a finite covering of open balls
$ V_{1},\ldots,V_{k} \subset M$ and constants $\sigma<1$ and $r>0$ such that
the restriction $f_{i}|_{B_r(V_i)}$ is a diffeomorphism onto its image and
\begin{equation*}\label{b}
||Df_{i}(x)^{-1}||<\sigma\ \  for\ every \ \   x\in B_r(V_{i}) \ \ and \ \  i=1,\ldots,k,
\end{equation*}
where $B_r(V_{i})$ stands for the $r$-neighborhood of the set $V_i$.
\item [(\textbf{C2})] The generators $f_{i}$, $i=1,\ldots,k$, are $B_r(V_i)$-conformal maps.
\end{enumerate}
\end{definition}
 \begin{figure}[htbp]
   \centerline{\includegraphics[scale=.2]{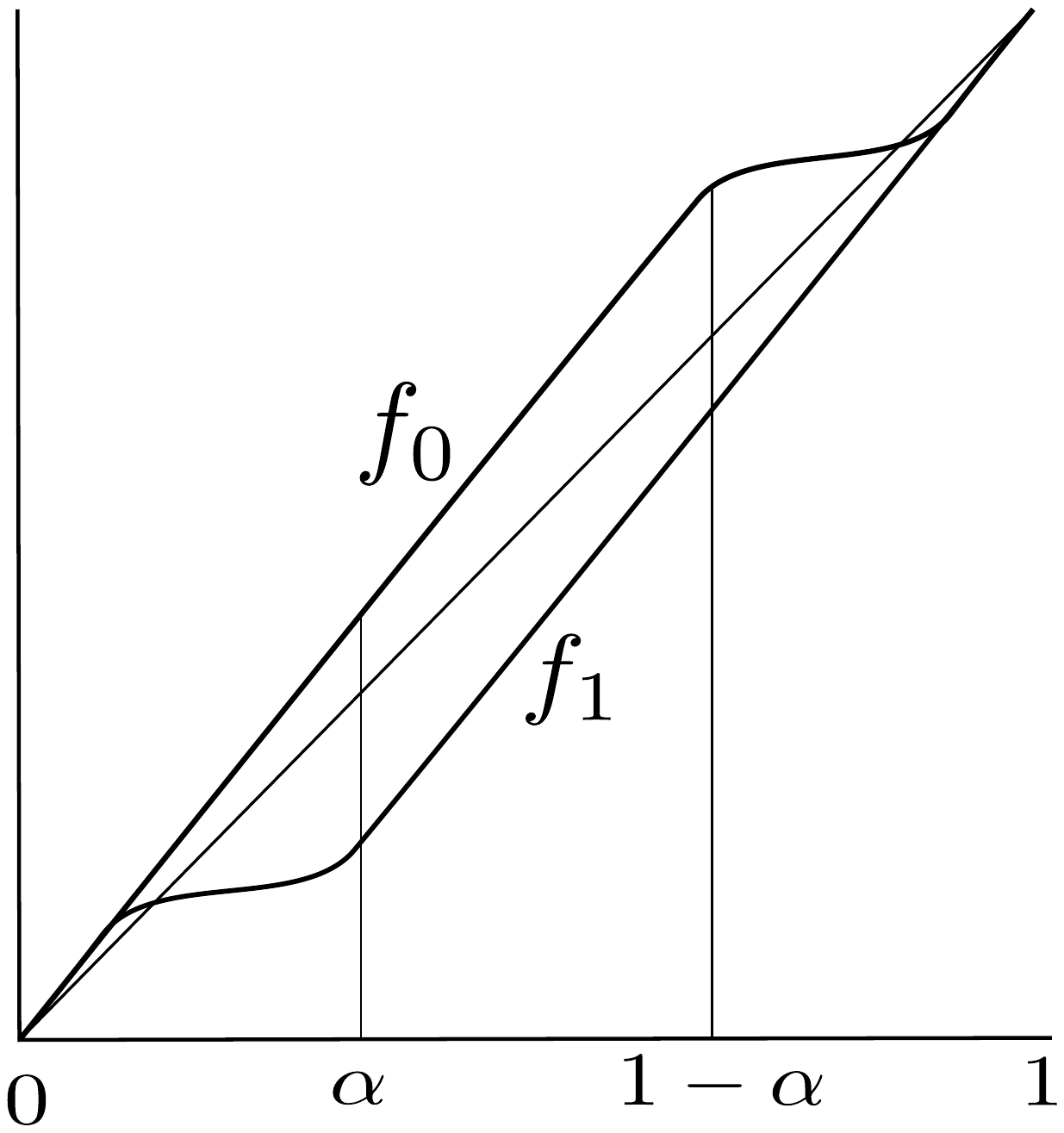}}
%
   \caption{Graph of $f_0$ and $f_1$}
   \label{Conformal Expanding}
   \end{figure}
As an example, suppose that $f_0$ and $f_1$ are two smooth diffeomorphisms on the unit circle $\mathbb{S}^1$ satisfying the following: (see Figure \ref{Conformal Expanding}).
both fix $0$,
$f_0(x)<x<f_1(x)$, for $x\in (0,1)$ and
for some $\alpha<1/2$, the derivatives
$f_0^\prime$ and $f_1^\prime$ are constant and expanding on $[0,1-\alpha]$ and $[\alpha,1]$ respectively,
Define four smooth diffeomorphisms $F_i$, $i=1,\ldots ,4$ on two dimensional torus
$\mathbb{T}^2$ by $F_1(x,y):=(f_0(x),f_0(y))$, $F_2(x,y):=(f_0(x),f_1(y))$, $F_3(x,y):=(f_1(x),f_0(y))$ and $F_4(x,y):=(f_1(x),f_1(y))$. Then, semigroup generated by $F_i$'s is locally conformal expanding.

\begin{definition}
Let $\Gamma$ be a semigroup of local diffeomorphisms defined on a compact smooth manifold $M$.
 A countable collection $\{M_i : i \in \mathbb{N}\}$ composed of closed subsets of $M$ together with
a collection $\{h_i \in \Gamma : i \in \mathbb{N}\}$ is called a \emph{countable Markov
partition} for $\Gamma$ if the followings hold:
\begin{enumerate}
\item [$\bullet$]~$\lambda(M \setminus \cup_{i \in \mathbb{N}} M_i)=0, \ \textnormal{and} \ M_i=\overline{M_i^\circ}.$
\item [$\bullet$]~$M_i^\circ \cap M_j^\circ=\emptyset \ \textnormal{whenever} \ i\neq j.$
\item [$\bullet$]~$\textnormal{If}$~ $M_i^\circ \cap h_i^{-1}(M_j^\circ) \neq \emptyset, \ \textnormal{then} \ h_i(M_i) \supset M_j.$
\end{enumerate}
A countable Markov partition $(\mathcal{M},\mathcal{H})=(M_i,h_i)_{i\in \mathbb{N}}$ has
the \emph{finite images property (FIP)} if $\mathcal{B}:=\{h_i(M_i^\circ): i \in \mathbb{N}\}$ is a finite covering, has $N$ elements say, of $M$.
The countable Markov partition has also \emph{finite cycle property (FCP)} if
there exists $n_1, \ldots, n_N \in \mathbb{N}$ such that for all $\ell \in \mathbb{N}$, there exist $i,j \in \{1, \ldots, N\}$,  such  that $h_{n_i}(M_{n_i})\supset M_{\ell}$ and $h_{\ell}(M_\ell)\supset M_{n_j}$.
\end{definition}
The exsistence of a countable Markov partition with FIP and FCP provide a inducing scheme of hyperbolic
type for a locally conformal expanding actions (see Subsection \ref{hypSc}, below, for precise statements). First, we should provide a condition under which a locally conformal expaanding semigroup has a countable Markov partiotion. The following notion which is a natural extension of topologically mixing in determinstic case to the locally conformal expanding semigroups, is necessary for this letter.

For a locally conformal expanding semigroup $\Gamma$,
take the covering $\{V_1, \ldots , V_k\}$ of $M$ and $r$ as in Definition \ref{expand} and choose $\eta<r/2$ as the Lebesgue number of this covering.
A finite word $w=(i_1, \ldots, i_n)$ with $i_j \in \{1, \ldots, k\}$ is an \emph{admissible word for $x$} if $x \in V_{i_1}$ and $f_{i_j}\circ \cdots \circ f_{i_1}(x)\in V_{i_{j+1}}$, for each $j=1, \ldots, n-1$. For a given pair of points $x,y\in M$, we say that $y$ is \emph{accessible by} $x$, if there exists an admissible word $w=( i_1, \ldots, i_n)$ for $x$ such that $f_w^n(x)=y$, where $f_w^n:=f_{i_n}\circ \cdots \circ f_{i_1}$, denoting by $x \vDash_w y$. Given a ball $B$ of radius $\varepsilon$ and center $y$, with $\varepsilon < \eta$, if $y$ is accessible by a point $x\in M$, along an admissible word $w=( i_1, \ldots, i_n)$, then the subset $B_{i_1, \ldots, i_n}:=f_w^{-n}(B)$ is called an admissible $n$-cylinder around $x$, where $f_w^{-n}(B)=(f_w^n)^{-1}(B)$. The semigroup $\Gamma$ is called a \emph{topologically mixing} semigroup if for each two small open sets $U$ and $V$ of $M$ with $\text{diam}(U)$ and $\text{diam}(V)$ are less than $\eta/2$, there exist an admissible word $w=(i_1, \ldots, i_n)$ and an open ball $B$ with radius $\varepsilon$, such that $B \supset V$ and the $n$-cylinder $B_{i_1, \ldots, i_n}=f_w^{-n}(B)$ is contained in $U$; hence $f_w^n(B_{i_1, \ldots, i_n})\supset V$ which implies that $f_w^n(U)\supset V$.
\begin{theorem A}\label{Theorem A}
Every topologically mixing locally conformal expanding semigroup admits a
countable Markov partition with FIP and FCP.
\end{theorem A}
\subsection{Inducing scheme of hyperbolic type for semigroups}\label{hypSc}
Pesin et al. in \cite{PSZ2} introduced an inducing scheme of hyperbolic type for a continuous map of a compact topological space
and described the associated tower construction to establish a thermodynamic formalism.
Here, we generalize this notion to semigroup actions and describe inducing schemes of hyperbolic type for finitely generated semigroups of local diffeomorphisms. For this, we need some preliminaries. For a countable Markov partition $(\mathcal{M},\mathcal{H})$ define a Markov chain as follows: let $ \mathcal{A}=(t_{ij})_{\mathbb{N} \times \mathbb{N}} $ be a matrix defined by
\begin{equation}\label{h}
t_{ij}:= \left\{
 \begin{array}{rl}
  1 & \text{if}\ M_{i}^\circ\cap h_i^{-1}(M_{j}^\circ)\neq\emptyset\\
  0 & \text{otherwise}.
 \end{array}\right.
\end{equation}
It is obvious that $\mathcal{A}=(t_{ij})_{\mathbb{N} \times \mathbb{N}}$ is a matrix of zeroes and ones and, by Markovian property, it has no zero rows.
Thus, $\mathcal{A}$ is a transition matrix.
Take
\begin{equation}\label{chain3}
\Sigma_{\mathcal{A}}^{(+)}:=\{\textbf{i}=(i_0, i_1, \ldots)\in \mathbb{N}^{\mathbb{Z}(\mathbb{N}\cup \{0\})}: \forall j\geq 0, \ \ t_{i_{j}i_{j+1}}=1\}
\end{equation}
the corresponding \emph{one-sided} and \emph{two-sided topological Markov chains} on the countable alphabet $\mathbb{N}$ and $\sigma_{\mathcal{A}}^{(+)}:\Sigma_{\mathcal{A}}^{(+)}\longrightarrow\Sigma_{\mathcal{A}}^{(+)}$
the \emph{left shift maps}. 
Equip $ \Sigma_{\mathcal{A}}^{(+)}$
with the topology generated by the base of cylinder sets
$C_{[a_{n},\ldots,a_{m}]}:=\{\textbf{i}\in\Sigma_{\mathcal{A}}^{(+)}:i_{j}=a_{j}$,
for $ j=n,\ldots,m\},$ where $n,m \in \mathbb{N}\cup \{0\}$ and $n\leq m$.
Now, suppose that $(\mathcal{M}, \mathcal{H})=\{(M_{i}, h_i)\}_{i\in\mathbb{N}}$ has FIP.
By the definition, there is a finite set $\mathcal{B}=\{B_1, \ldots, B_N\}$ of open sets such that for each $i \in \mathbb{N}$, $h_i(M_i^\circ)\in \mathcal{B}$, for each $j=1, \ldots, N $.
Put $\mathcal{M}^\circ:=\{M_i^\circ\}$. Let $|h_i|$ be the length of the admissible finite word $w_i$, $i \in\mathbb{N}$.
Take the subset $W := \bigcup_{M_i \in \mathcal{M}}M_i^\circ$ which is called the \emph{inducing domain}.
Define the \emph{inducing time} $\tau : M \to \mathbb{N}$ by
\begin{equation}\label{tau}
\tau(x):= \left\{
 \begin{array}{rl}
  |h_i| & \text{if}\ x \in M_{i}^\circ\\
  0 & \text{if} \ x \notin W,
 \end{array}\right.
\end{equation}
That is $h_i=f_{w_i}^{\tau}$. We say that $\Gamma$ admits an \emph{inducing scheme of hyperbolic
type} $\{\mathcal{M},\tau\}$, with inducing domain $W $
and inducing time $\tau $
provided the two following conditions hold:
\begin{enumerate}
  \item [(\textbf{H1})] For any $M_i \in \mathcal{M}$, $h_i|_{M_i^\circ}$ can be extended to a homeomorphism of a neighborhood of $M_i$.
  \item [(\textbf{H2})] For every bi-infinite sequence  $\textbf{i}=(\ldots, i_{-1}, i_0, i_1, \ldots)\in\Sigma_{\mathcal{A}}$ there exists a unique
  sequence $\underline{x}=\underline{x}(\textbf{i})=(x_n)_{n \in \mathbb{Z}}$ such that
  \begin{enumerate}
    \item [(a)] $x_n \in M_{i_n}$ and $h_{i_n}( x_n)= x_{n+1}$,
    \item [(b)] if for some sequences $\underline{x}=\underline{x}(\textbf{i})$ and $\underline{y}=\underline{y}(\textbf{j})$,
    one has $x_n=y_n$ for all $n \leq 0$ then $(\ldots, i_{-1}, i_0)=(\ldots, j_{-1}, j_0)$.
  \end{enumerate}
\end{enumerate}
Define the \emph{induced map} $F :W \to M$ by
\begin{equation}\label{ind}
F(x):=h_i(x),~\text{for}~x \in M_i^\circ.
\end{equation}
For each $\textbf{i} \in \Sigma_{\mathcal{A}}$, consider the sequence $\underline{x}(\textbf{i})$ given by condition (\textbf{H2}) and
define the \emph{coding map} $\pi : \Sigma_{\mathcal{A}} \to M$ by
\begin{equation}\label{code}
\pi(\textbf{i}):=x_0.
\end{equation}
Let
\begin{equation}\label{hat}
 \check{\Sigma}:=\{\textbf{i}=(\ldots, i_{-1}, i_0, i_1, \ldots) \in \Sigma_{\mathcal{A}} : x_n( \textbf{i}) \in M_{i_n}^\circ, \ \text{for~all} \ n \in \mathbb{Z}\}.
\end{equation}
We additionally require that the two following conditions holds for inducing scheme $\{\mathcal{M},\tau\}$. In the rest, we will prove the existence and uniqueness of equilibrium measures using these two imposed conditions.
\begin{enumerate}
  \item [(\textbf{H3})]If $\mu$ is an ergodic $\sigma_{\mathcal{A}}$-invariant measure of full support then $\mu(\Sigma_{\mathcal{A}}\setminus  \check{\Sigma}) = 0$.
  \item [(\textbf{H4})] The induced map $F$ has at least one periodic point in $W$.
  \end{enumerate}
Using condition (\textbf{H3}), we show that every Gibbs measure is
supported on $\check{\Sigma}$ and its projection by the coding map $\pi$ is thus supported on $W$ and is $F$-invariant.
The projected measure is a natural candidate for the equilibrium measure for $F$. Condition
(\textbf{H4}) will be used to prove that the pressure function is finite (see Proposition \ref{p1}).

\begin{theorem B}\label{B}
Each locally conformal expanding topologically
mixing semigroup of $C^{1+\alpha}$ locally conformal local diffeomorphisms on a compact connected manifold $M$
exhibits an inducing scheme $(\mathcal{M},\tau)$ of hyperbolic type satisfying conditions \normalfont{\textbf{(H1)-(H4)}}.
\end{theorem B}

\subsection{Inducing scheme for the induced skew product}
We consider locally conformal semigroup actions and explore the relation between ergodic properties
 of semigroup actions and their induced skew products.
Locally conformal semigroups exhibit inducing schemes of hyperbolic type.
The tower construction of these semigroup actions allows us to define an inducing scheme for their induced skew products.
We introduce liftable measures and describe those measures that give positive weight to the base of the tower and can be lifted.
We use the symbolic representation by the tower and apply the results in Section \ref{section6}
to establish a thermodynamic formalism for liftable measures.

For a locally conformal expanding semigroup $\Gamma$, let $\{V_1, \ldots , V_k\}$ and $r$ be as in Definition \ref{expand}. Suppose that finite matrix $ A;=(a_{ij})_{i,j=1}^k $ is defined by
\begin{equation}\label{h1}
a_{ij}:= \left\{
 \begin{array}{rl}
  1 & \text{if}\ f_i(V_i)\cap V_j\neq \emptyset\\
  0 & \text{otherwise}.
 \end{array}\right.
\end{equation}
We may assume that $\max \{\text{diam}(V_i): i=1, \ldots, k\}<r$. Put $N_k:=\{1, \ldots, k\}$.
Since $\{V_1, \ldots, V_k\}$ is a covering of $M$, $A$ has no zero rows saying that $A$ is actually a transition matrix.
In particular, if $\Gamma$ is topologically mixing then $A$ has no zero columns.
Similar to the case of $\mathcal{A}$, we define
\begin{equation}\label{chain}
\Sigma_{A}^{(+)}:=\{\textbf{j}=(j_0, j_1, \ldots)\in N_k^{\mathbb{Z}(\mathbb{N}\cup \{0\})},~a_{j_{\ell}j_{\ell+1}}=1\}
\end{equation}
the corresponding \emph{one-sided and two-sided finite topological Markov chains}, respectively, and $\sigma_{A}^+:\Sigma_{A}^+\longrightarrow\Sigma_{A}^+$
and $\sigma_{A}:\Sigma_{A}\longrightarrow\Sigma_{A}$
the \emph{left shift maps}.
We equip $ \Sigma_{A}^{+}$ and $ \Sigma_{A}$ with the
topology generated by the base of cylinder sets
$C_{[a_{n},\ldots,a_{m}]}:=\{\textbf{i}\in\Sigma_{A}^{(+)}:i_{j}=a_{j}$,
for $ j=n,\ldots,m\}$, where $n,m \in \mathbb{Z}(\mathbb{N}\cup \{0\})$ and $n\leq m$.
Markov chain $(\Sigma_{A}, \sigma_{A})$ induces a skew product of the form
\begin{equation}\label{s2}
F_{A}:\Sigma_A\times M \to \Sigma_A\times M, \ F_{A}(\textbf{i},x):=(\sigma_{A}\textbf{i}, f_{i_0}(x)),
\end{equation}
where $i_0$ is the digit at the
zero position of $\textbf{i}=(\ldots, i_{-1},i_0,i_1,\ldots)\in \Sigma_A$. The skew product $F_{A}$ is called the \emph{induced skew product} of the semigroup $\Gamma$. Using the inducing scheme of hyperbolic type $(\mathcal{M},\tau)$, we define an inducing scheme for the induced skew product $F_A$ as follows.

For each $j \in\mathbb{N}$, consider the finite word $w_j=(i_j^1,\ldots, i_j^{s_j})$, with $s_j=\tau(M_j^\circ)$ such that $h_j=f_{i_j^{s_j}}\circ\cdots\circ f_{i_j^{1}}$.
Then, we define an injective map $\rho:\Sigma_\mathcal{A}\to \Sigma_A$ such that it takes each branch $(\ldots, i_{-1}, i_0, i_1, \ldots)\in \Sigma_{\mathcal{A}}$ to the concatenation of associated finite words, that is
\begin{equation}\label{rho}
\rho(\ldots, i_{-1}, i_0, i_1, \ldots):=(\ldots ,i_{-1}^1,\ldots, i_{-1}^{s_{-1}},i_{0}^1,\ldots, i_{0}^{s_{0}},  \ldots).
\end{equation}
Let us take $C(i)\subset \Sigma_\mathcal{A}$ the cylinder induced by the letter $i \in\mathbb{N}$ and put
\begin{equation}\label{w}
 \mathbb{W}:=\bigcup_{i \in\mathbb{N}}\rho(C(i))\times M_i \subset \Sigma_{A}\times M.
\end{equation}
We now define measurable maps $\psi:\mathbb{W} \to \mathbb{N}\cup \{0\}$ by
\begin{equation}\label{psi}
 \psi(\textbf{i},x):=\left\{
\begin{array}{rl}
\tau(x) & \text{if} \ \ x\in M_{i_0}^\circ\\
0 & \text{otherwise},
\end{array}\right.
\end{equation}
and
$F_A^\psi: \mathbb{W} \to  \mathbb{W}$ by
\begin{equation}\label{s3}
 F_A^\psi(\textbf{i},x):=\left\{
\begin{array}{rl}
(\sigma_A^{\tau(x)}(\textbf{i}),f_{\textbf{i}}^{\tau(x)}(x)) & \text{if} \ \ x\in M_{i_0}^\circ\\
(\textbf{i},x) & \text{otherwise},
\end{array}\right.
\end{equation}
where $i_0$ is the zeroth digit of $\rho^{-1}(\textbf{i})$. Mappings $\psi$ and $F_A^\psi$ are called \emph{inducing time} and \emph{induced map}, respectively.
Let
$
 \mathbb{Y}:=\{F_A^k(\textbf{i},x): (\textbf{i},x)\in \mathbb{W} \ \text{and} \ 0\leq k \leq \psi(\textbf{i},x)-1\},
$
where $\psi$ is the inducing time given by (\ref{psi}). Then, $\mathbb{Y}$ is forward invariant under $F_A$.
Set
 $\mathcal{S}:=\{\rho(C(i))\times M_i: i \in \mathbb{N}\}$.
Define the induced map $\mathbb{F}:\mathbb{W} \to \mathbb{W}$ by
\begin{equation}\label{imap}
\mathbb{F}|_{\rho(C(i))\times M_i}:= F_A^\psi |_{\rho(C(i))\times M_i}, \ i \in \mathbb{N}.
\end{equation}
Then $(\mathcal{S},\psi)$ is an inducing scheme for $F_A$.
\begin{remark}\label{rem222}
It is easy to see that $\textbf{i} \in \check{\Sigma}$ is a periodic orbit for the shift $\sigma_{\mathcal{A}}$
if and only if $(\rho(\textbf{i}),\pi(\textbf{i}))$ is a periodic orbit for $\mathbb{F}$, where the subset $\check{\Sigma}\subset \Sigma_\mathcal{A}$ defined by (\ref{hat}).
\end{remark}
Denote the set of $F_A$-invariant ergodic Borel probability measures on $\mathbb{Y}$ by $\mathcal{M}(F_A, \mathbb{Y})$ and
the set of $\mathbb{F}$-invariant ergodic Borel probability measures on $\mathbb{W}$ by $\mathcal{M}(\mathbb{F}, \mathbb{W})$.
For any $m \in \mathcal{M}(\mathbb{F}, \mathbb{W})$, let
\begin{equation}\label{q1}
Q_{m}:=\int_{\mathbb{W}}\psi dm.
\end{equation}
If $Q_m < \infty$ the \emph{lifted measure} $\mathcal{L}(m)$ is defined as follows: for any $E \subset \mathbb{Y}$
\begin{equation}\label{lift-m1}
\mathcal{L}(m)(E):=\frac{1}{Q_{m}}\sum_{i\in\mathbb{N}}\sum_{k= 0}^{\psi(\rho(C(i))\times M_i)-1}m( F_A^{-k}(E)\cap \rho(C(i))\times M_i).
\end{equation}
Consider the class of measures
\begin{equation*}
  \mathcal{M}_L(F_A,\mathbb{Y}):=\{\mu \in \mathcal{M}(F_A, \mathbb{Y}): \text{there is} \ m\in \mathcal{M}(\mathbb{F}, \mathbb{W}) \ \text{with} \ \mathcal{L}(m)=\mu\}.
\end{equation*}
A measure $\mu \in \mathcal{M}_L(F_A,\mathbb{Y})$ is called \emph{liftable} and $\mathcal{L}^{-1}\mu$ is an \emph{induced measure} for $\mu$. Note that by \cite[Theorem 1.1]{Z}) if the inducing time $\psi \in L^1(\mathbb{Y},\mu)$, then $\mu \in  \mathcal{M}_L(F_A,\mathbb{Y})$.
\begin{theorem C}\label{Thm: C}
There exists an invariant ergodic probability measure $m$ for $F_A^{\psi}$ whose support is contained in $\mathbb{W}$.
Furthermore, $Q_m <\infty$ and the measure $m$ can be lifted to an $F_A$-invariant measure.
\end{theorem C}
\subsection{Equilibrium measures and Potentials}
Let $\phi: \Sigma_A\times M \to \mathbb{R}$ be a Borel function. We assume
that $\phi$ is well-defined and is finite at every point $(\textbf{i},x) \in \Sigma_A\times M$ and we call $\phi$ a \emph{potential}. For a potential $\phi$, a measure $\mu_\phi$, in the space $\mathcal{M}_L(F_A,\mathbb{Y})$ of liftable measures, called an \emph{equilibrium measure} for
$\phi$ if
\begin{equation*}
  P_L(\phi):=\sup_{\mu \in \mathcal{M}_L(F_A,\mathbb{Y})}\{h_\mu(F_A)+\int_{\Sigma_{A}\times M} \phi d\mu\}=h_{\mu_\phi}(F_A)+\int_{\Sigma_{A}\times M} \phi d\mu_\phi.
\end{equation*}
This definition of equilibrium measures only considers liftable measures and it differs from the standard one. We also define the induced potential $\overline{\phi}: \mathbb{W} \to \mathbb{R}$ by
\begin{equation}\label{lpsi}
\overline{\phi}(\textbf{i},x) :=\sum_{\ell=0}^{\psi(\rho(C(\textbf{i}))\times M_i)-1}\phi (F_A^\ell(\textbf{i},x)), \ (\textbf{i},x)\in \rho(C(\textbf{i}))\times M_i,
\end{equation}
where $\rho$ and $\psi$ are given by (\ref{rho}) and (\ref{psi}), respectively.
For the induced potential $\overline{\phi}$ given by (\ref{lpsi}), we take $\Phi:\Sigma_\mathcal{A} \to \mathbb{R}$ by
$$\Phi(\textbf{i}):=\overline{\phi}(\rho(\textbf{i}), \pi(\textbf{i})),$$
where $\pi$ is the coding map given by (\ref{code}).
Clearly, $ \Phi$ is well-defined on $\Sigma_\mathcal{A}$.
We call a measure $\nu_{\overline{\phi}}$ on $\mathbb{W}$ a Gibbs measure for $\overline{\phi}$ if the measure $(\rho \times \pi)^{-1}_\ast \nu_{\overline{\phi}}$
is a Gibbs measure for $\Phi$.
We call $\nu_{\overline{\phi}}$ an equilibrium measure for $\overline{\phi}$ if $- \int_{\mathbb{W}}\overline{\phi} d\nu_{\overline{\phi}}< \infty$ and
$$h_{\nu_{\overline{\phi}}}(\mathbb{F})+\int_{\mathbb{W}} \overline{\phi} d\nu_{\overline{\phi}} =\sup_{\nu \in \mathcal{M}(\mathbb{F},\mathbb{W}):  \ - \int_{\mathbb{W}}\overline{\phi} d\nu< \infty}\{h_\nu(\mathbb{F})+\int_{\mathbb{W}} \overline{\phi } d\nu \}.$$
For $\ell \in \mathbb{N}$, let
\begin{equation*}\label{z}
 Z_n(\Phi, \ell):=\sum_{\sigma_{\mathcal{A}}^n(\textbf{i})=\textbf{i}, i_0=\ell}\exp (\Phi_n(\textbf{i})),
\end{equation*}
where $\Phi_{n}(\textbf{i})=\sum_{k=0}^{n-1}\Phi \circ \sigma_{\mathcal{A}}^{k}(\textbf{i})$.
The \emph{Gurevic} pressure of $\Phi$ is defined by
$$\text{P}_G(\Phi):=\lim_{n \to \infty}\frac{1}{n} \log Z_n(\Phi, \ell).$$
This notion introduced by Gurevic in \cite{G1, G2} which is a generalization of the notion of topological entropy
$h_G(\sigma_{\mathcal{A}})$ for countable Markov chains, so that $P_G(0) = h_G(\sigma_{\mathcal{A}})$.
For a potential $\Phi:\Sigma_{\mathcal{A}} \to \mathbb{R}$ we denote by
\begin{equation*}\label{pot}
\text{Var}_n(\Phi):=\sup\{|\Phi(\textbf{i})-\Phi(\textbf{j})|: \textbf{i},\textbf{j} \in \Sigma_{\mathcal{A}}, \ i_\ell=j_\ell, \ -n+1 \leq \ell \leq n-1\}
\end{equation*}
the $n$-th \emph{variation} of $\Phi$. Since the Gurevich pressure only depends on the positive side of the sequences,
$ P_G(\Phi)$ exists whenever $\sum_{n=1}^\infty n \text{Var}_n(\Phi)< \infty$
 and it is independent of $\ell \in \mathbb{N}$ by Theorem 1 of \cite{Sa1}. Following \cite{PSZ2}, the potential $\overline{\phi}$
\begin{enumerate}
\item [(a)] has \emph{summable variations} if the function $\Phi$ has summable variations, i.e.,
$$\sum_{n \geq 1} \text{Var}_n(\overline{\phi} \circ (\rho \times \pi)):=\sum_{n \geq 1}\text{Var}_n(\Phi)< \infty,$$
\item [(b)] has \emph{finite Gurevich pressure} if $\text{P}_G(\overline{\phi} \circ (\rho \times \pi))=\text{P}_G(\Phi)< \infty$.
\end{enumerate}
Let us take
$\phi^+:=\overline{\phi - \text{P}_L(\phi)}=\overline{\phi}-\text{P}_L(\phi)\psi$
and let $\Phi^+:= \phi^+ \circ (\rho \times \pi)$ with $\rho \times \pi(\textbf{i})=(\rho(\textbf{i}),\pi(\textbf{i}))$, for $\textbf{i}\in \Sigma_{\mathcal{A}}$.
We also impose some additional condition to potential $\phi$:
 \begin{enumerate}
   \item [(\textbf{P1})] there exist $C > 0$ and $0 < r < 1$ such that for any $n \geq 1$, $\text{Var}_n(\Phi) \leq C r^n$,
   \item [(\textbf{P2})] $$\sum_{\rho(C(\textbf{i}))\times M_i}\sup_{(\textbf{i},x) \in \rho(C(\textbf{i}))\times M_i}\exp \overline{\phi}(\textbf{i},x)<\infty,$$
   \item [(\textbf{P3})] there exists $\varepsilon > 0$ such that
   $$\sum_{\rho(C(\textbf{i}))\times M_i}\psi(\rho(C(\textbf{i}))\times M_i)\sup_{(\textbf{i},x) \in \rho(C(\textbf{i}))\times M_i}\exp (\phi^+(\textbf{i},x)+\varepsilon \psi(\textbf{i},x))<\infty. $$
 \end{enumerate}
  These conditions considered in \cite{PSZ2} for deterministic systems.
The next result shows that the induced skew product $F_A$ possesses a unique equilibrium measure $\mu_\phi$, associated to each $\phi$ satisfying Conditions
(\textbf{P1})-(\textbf{P3}), which minimizes the free energy among the measures that are liftable to the tower.
\begin{theorem D}\label{D}
Let $(\mathcal{S},\psi)$ be an inducing scheme for the induced skew product $F_A$ corresponding to the inducing scheme of hyperbolic
type $(\mathcal{M},\tau)$ for a topologically mixing locally conformal expanding semigroup $\Gamma$. Assume that the potential function $\phi$ satisfies Conditions
$(\normalfont{\textbf{P}}1)-(\normalfont{\textbf{P}}3)$. Then there exists a unique equilibrium measure $\mu_\phi$ for $\phi$ among all measures in
$\mathcal{M}_{L}(F_A,\mathbb{Y})$. In particular, $\mu_\phi$ is ergodic.
\end{theorem D}

\hspace{-.4cm}\textbf{Organization.}
Let us give an overview of the content of the paper.
In Subsection \ref{sub3}, we prove Theorem A and construct a special countable Markov partition for topologically mixing locally conformal expanding semigroup actions which is
achieved primarily through a finite cover and satisfies the finite cycle property.
We will obtain a bounded distortion formulation for the induced map $F$ in Subsection \ref{section4}.
As a consequence, we investigate the existence of a $F$-invariant absolutely continuous probability measure.
In Subsection \ref{section5}, we show that each locally conformal topologically mixing semigroup of $C^{1+\alpha}$ conformal (local) diffeomorphisms
on a compact connected manifold $M$ exhibits an inducing scheme of hyperbolic type satisfying conditions (\textbf{H1})-(\textbf{H4}), this proves Theorem B.
Our goal in Subsection \ref{section6} is to carry out the thermodynamic formalism
for countable subshift $\sigma_{\mathcal{A}}$ on the space $\Sigma_{\mathcal{A}}$ of two-sided infinite sequences.
Note that for locally conformal semigroups, the semigroup action may be
understood either as a random walk.
It ensures to provide a link between the dynamics of the semigroup action and the induced skew product $F_{A}$ given by (\ref{s2}).
The tower construction of locally conformal semigroup actions
allows us to define an inducing scheme for their induced skew products.
Then, we apply the notion of random Markov fibred system introduced in \cite{DKS} to construct a suitable Markov measure in Section \ref{Section: 5.1}.
This yields the set of liftable measures of induced skew products is nonempty.
In Section \ref{section8}, we solve the liftability problem for towers associated with inducing
schemes of induced skew products and prove Theorem C.
By Zweim\"{u}ller \cite{Z}, a measure is liftable to the tower if it has integrable inducing time.
Finally, using the Zweim\"{u}ller's result and Abramov's and Kac's formula, we provide a crucial relation between the thermodynamic
formalism of original and induced systems.
Theorem D is proved in Section \ref{section8}.
\section{Countable Markov partitions and ergodicity of induced map}\label{section3}
Dynamical systems that admit Markov partitions with finite or
countable number of partition elements allow symbolic representations by
topological Markov shifts with finite or countable alphabets. As
a result these systems exhibit high level of chaotic behaviors of trajectories.
In this section we will construct a countable Markov partition with the finite images (FIP) and the finite cycle (FCP) properties for locally conformal semigroups.
Moreover, we use the
countable Markov partition to acquire a bounded distortion formulation for the induced map $F$ given by (\ref{ind}).
\subsection{Countable Markov partitions associated to locally conformal actions}\label{sub3}
Our aim here is to construct a countable Markov partition for locally conformal semigroups which satisfies FIP and FCP.
Let us fix a locally conformal expanding semigroup $\Gamma$ generated by a finite family $f_1, \ldots, f_k$ of $C^{1+\alpha}$ local diffeomorphisms on a compact manifold $M$
satisfying conditions $\textbf{\text{(C1)}}$ and $\textbf{\text{(C2)}}$ in Definition \ref{expand}. Take $\eta>0$ the Lebesgue number of the covering $\{V_1, \ldots , V_k\}$. Assume $\eta$ is small enough so that $\eta < \frac{r}{2}$.
Suppose that $w=(i_0, \ldots, i_{n-1})$ is an admissible word for $x$ of the length $n$.
For any $0 < \varepsilon < \frac{\eta}{2}$ the \emph{dynamical} $n$-\emph{ball} $B(x,n,w,\varepsilon)$ defined as follows
\begin{equation}\label{01}
B(x,n,w,\varepsilon):=\{y \in M: d(f_w^j(x),f_w^j(y))<\varepsilon, \ \text{for \ every} \ 0\leq j \leq n\}.
\end{equation}
Recall that $f_w^n=f_{i_{n-1}}\circ \cdots \circ f_{i_0}$.
\begin{lemma}\label{lem0}
Take a dynamical $n$-ball of the form (\ref{01}). Then $f_w^{n}(B(x,n,w,\varepsilon))=B(f_w^{n}(x),\varepsilon)$, for any admissible $n$-word of $x$.
\end{lemma}
\begin{proof}
Obviously, by definition of a dynamical ball and due to condition $\textbf{\text{(I1)}}$, we get
\begin{equation}\label{in}
d(f_w^j \circ f_w^{-n}(z), f_w^j \circ f_w^{-n}(u)) \leq \sigma^{n-j}d(z,u)
\end{equation}
for every $z,u \in B(f_w^n(x), \varepsilon)$ and every $0 \leq j \leq n$, where $f_w^{-n}=(f_w^n)^{-1}$. Now, the
inclusion $f_w^n(B(x,n,w,\varepsilon))\subseteq B(f_w^n(x),\varepsilon)$ is an immediate consequence of definition of a dynamical ball.
To prove the converse, consider the inverse map $f_w^{-n}:B(f_w^n(x), \varepsilon) \to B(x,\varepsilon)$.
Given any $y \in B(f_w^n(x),\varepsilon)$, let $z = f_w^{-n}(y)$. Then, $f_w^n(z) = y$ and, by (\ref{in}), for every $0 \leq j \leq n$,
\onehalfspacing
\begin{equation*}
d(f_w^j(z), f_w^j (x)) \leq \sigma^{n-j}d(f_w^n(z), f_w^n(x)) \leq d(y, f_w^n(x)) < \varepsilon.
\end{equation*}
This shows that $z \in B(x,n,w,\varepsilon)$.
\end{proof}

Given any path-connected domain $D \subseteq M$, we define the inner
distance $d_D(x, y)$ between two points $x$ and $y$ in $D$ to be the infimum of the
lengths of all curves joining $x$ to $y$ inside $D$. For any curve $\gamma$ inside $D$, we denote by $d_D(\gamma)$ the length of $\gamma$.
For a ball $B$, the notation $\mathcal{D}(B)$ stands for the set of all diameters of $B$.
We recall that a diameter of $B$ is a geodesic curve $\gamma$ inside $B$ which connects two points $x$ and $y$ from the boundary $\partial B$.
We need the following auxiliary lemma.

\begin{lemma}\label{lem1}
There is $K>0$ such that for any open ball $B=B(y,\varepsilon) $, with $\varepsilon \leq \frac{\eta}{2}$, any $n \in \mathbb{N}$
and any admissible $n$-word $w=(i_1, \ldots, i_{n})$ such that $B$ is accessible by an $n$-ball $B_{n}=B(x,n,w,\varepsilon)$ around the point $x$,
one has that
\onehalfspacing
\begin{equation*}
\min_{\gamma \in\mathcal{D}(B)}d_{B_{n}}(\gamma_n)\geq K \max_{\gamma \in\mathcal{D}(B)}d_{B_{n}}(\gamma_n),
\end{equation*}
where, for each curve $\gamma \in\mathcal{D}(B)$, $\gamma_n$ is the lifting of $\gamma$ inside $B(x,n,w,\varepsilon)$, i.e. $f_w^n(\gamma_n)=\gamma$.
\end{lemma}
\begin{proof}
For any $z\in B$, take the points $z_j$, $1\leq j\leq n$, such that $z_n=z$, $z_{j-1}=f_{i_j}^{-1}(z_j)$ and let $a_{i_j}(z_j):=m(Df_{i_j}^{-1}(z_j))=\|Df_{i_j}^{-1}(z_j)\|$. By conformality of the mappings $f_j |_{B_r(V_j)}$, $j=1, \ldots, k$, one has that
$m(D f_w^{-n}(z))=\|D f_w^{-n}(z)\|=\prod_{j=1}^n a_{i_j}(z_j)$. Take $a_{( i_1, \ldots, i_n)}(z):=\prod_{j=1}^n a_{i_j}(z_j)$, we conclude
\onehalfspacing

\begin{equation*}\label{3e}
\frac{\max_{\gamma \in\mathcal{D}(B)}d_{B_{n+1}}(f_w^{-n}\gamma)}{\min_{\gamma \in\mathcal{D}(B)}d_{B_{n+1}}(f_w^{-n}\gamma)}\leq \frac{\sup\{a_{( i_1, \ldots, i_n)}(z): z \in B\}}{\inf\{a_{(i_1, \ldots, i_n)}(z): z \in B\}}.
\end{equation*}
Let $F_j(x):=\log (a_j(x))$, for $x \in f_j(B_r(V_j))$ and $j=1, \ldots, k$, where $a_j(x)=m(Df_{j}^{-1}(x))$. Since $f_j$ is $C^{1+\alpha}$, the mapping $F_j$ is H\"{o}lder continuous, hence there exist constants $C\geq 1$ and $\alpha >0$ such that for any $x,y \in f_j(B_r(V_j))$, $j=1, \ldots, k$, one has that
$|F_j(x)-F_j(y)|\leq C d(x,y)^\alpha. $
Hence, for every $z,u \in B$, taking the points $z_j$ and $u_j$ with $z_n=z$, $u_n=u$, $z_{j-1}=f_{i_j}^{-1}(z_{j})$ and $u_{j-1}=f_{i_j}^{-1}(u_{j})$, $j=1, \ldots, n$, one has that
\begin{align*}
\log a_{( i_1, \ldots, i_n)}(z)-\log a_{( i_1, \ldots, i_n)}(u)|
&=|\sum_{j=1}^n \log a_{i_j}(z_j)- \sum_{j=1}^n \log a_{i_j}(u_j)|\\
&\leq  \sum_{j=1}^n |\log a_{i_j}(z_j)- \log a_{i_j}(u_j)|\\
&=\sum_{j=1}^n |F_{i_j}(z_j)-F_{i_j}(u_j)|\\
&\leq C \sum_{j=1}^n d(z_j,u_j)^\alpha \\
&\leq  C(2\rho)^\alpha \sum_{j=1}^n \sigma^{{(n-j)} \alpha}\\
&\leq  \frac{C(2\rho)^\alpha}{1-\sigma^{\alpha}},
\end{align*}
where $\rho=\max\{\text{diam}(f_j(B_r(V_j))): j=1, \ldots, k\}$.
It is enough to take $K:=\exp{(\frac{C(2\rho)^\alpha}{1-\sigma^{\alpha}})}$.
\end{proof}
Let us take a finite covering $\mathcal{B}=\{B_1, \ldots, B_N\}$ of open balls of $M$ with centers $p_j$, $j=1, \ldots, N$ and radius $\varepsilon$ with $0 < \varepsilon \leq \eta/6$,
where $\eta < r/2$ is the Lebesgue number of the covering $\{V_1, \ldots , V_k\}$. If the point $p_j$ is accessible by a point $x\in M$ along an admissible word $w=( i_1, \ldots, i_n)$ of $x$ then by Lemma \ref{lem0}, $f_w^{-n}(B_j)=B(x,n,w,\varepsilon)$, where $B(x,n,w,\varepsilon)$ is a dynamical $n$-ball around $x$. Put $M_{n,j,w}(x):=\text{Cl}(f_w^{-n}(B_j))$. If there is no confusion, we will abbreviate $M_{n,j,w}(x)$ by $M_{n,j,w}$. Briefly, we say that $M_{n,j,w}$ is a closed dynamical ball which is \emph{achieved through} the collection $\mathcal{B}$ and denote it by $M_{n,j,w}\vDash B_j$.
Let us take AD the set of all finite admissible words.
\begin{proposition}\label{pro1}
The set
\begin{equation*}\label{v}
\mathcal{V}:=\{M_{n,j,w}: n\in \mathbb{N}, j\in \{1, \ldots, N\}, \ w\in \normalfont{AD} \ \text{with} \ M_{n,j,w}\vDash B_j\}
\end{equation*}
of all closed dynamical balls achievable through the collection $\mathcal{B}$ satisfies the following conditions:
\begin{enumerate}
\item [$(1)$] $\mathcal{V}$ is a covering of $M$;
\item [$(2)$] for every $x \in M$ and $\delta > 0$, there is an element $M_{n,j,w}\in \mathcal{V}$ such that $x \in M_{n,j,w}$ and $\text{diam}(M_{n,j,w})<\delta$;
\item [$(3)$] there is a constant $L_1 >0$, does not depend on the indices $n$, $j$ and the word $w$, such that for any $M_{n,j,w}\in \mathcal{V}$, $\text{diam}(M_{n,j,w})^d \leq L_1 \lambda(M_{n,j,w})$, where $d=\text{dim}(M)$ and $\lambda$ denotes the normalized Lebesgue measure on $M$.
\end{enumerate}
\end{proposition}
Each covering $\mathcal{V}$ exhibiting all conditions (1), (2) and (3) of Proposition \ref{pro1} is called a \emph{Vitali covering} of $M$.
\begin{proof}
We first claim that for each $\ell \in \mathbb{N}$, the collection
\begin{equation}\label{vn}
\mathcal{V}_\ell:=\{M_{n,j,w}\in \mathcal{V}: n \geq \ell, \ j=1, \ldots, N \}
\end{equation}
is a covering of $M$. Indeed, for each $x\in M$ and for any $n \geq \ell$, there exists an admissible word $w\in AD$ of the length $n$.
Since $\mathcal{B}$ is a covering of $M$ there is $j \in \{1, \ldots, N\}$, $f_w^n(x) \in B_j$ and hence, $x\in f_w^{-n}(B_j)$, that is $x$ is contained in $M_{n,j,w}$. Now, since $\mathcal{V}=\bigcup_{\ell=1}^\infty \mathcal{V}_\ell$, so the first condition holds. Condition (2) is an immediate consequence of inequality (\ref{in}) given in Lemma \ref{lem0}.
Finally, the last condition is a consequence of Lemma \ref{lem1}.
\end{proof}
By the classical Vitali Theorem \cite{EG} if $\mathcal{V}$ is a Vitali covering
for a set $A \subset \mathbb{R}^m$, then there is in $\mathcal{V}$
a sequence of pairwise disjoint elements whose union exhausts all
of $A$ but a Lebesgue null set.
The next result is a reformulation of Vitali Covering Theorem in our context.
\begin{proposition}\label{pro22}
Under the above assumptions, for any open set $U$ of $M$, there exists a finite or countably infinite disjoint subcollection $\{M_j\}\subseteq \mathcal{V}$ such that $\lambda(U\Delta \cup_{j}M_j)=0$. 
\end{proposition}
Now we show that the semigroup $\Gamma$ admits a countable Markov partition with the finite images and finite cycle properties if it is topologically mixing.
\begin{theorem}\label{cmarkov}
Every topologically mixing locally conformal expanding semigroup $\Gamma$ satisfying conditions $\normalfont{\textbf{(C1)}}$ and $\normalfont{\textbf{(C2)}}$
admits a countable Markov partition with \normalfont{FIP} and \normalfont{FCP}.
\end{theorem}
\begin{proof}
The collections $\mathcal{V}$ and $\mathcal{V}_\ell$, provided by Proposition \ref{pro1}, are Vitali coverings of $M$.
Take $\ell$ large enough so that $\text{diam}(\mathcal{V}_\ell)<\beta$, where $\beta$ is the Lebesgue number of the covering $\mathcal{B}$.
Since $\Gamma$ is topologically mixing, hence for each $j=1, \ldots, N$, there is a closed dynamical ball $M_{n_j,i_j,w_j} \in \mathcal{V}_\ell$, with $n_j \in \mathbb{N}$,
$i_j \in \{1, \ldots, N\}$ and $w_j \in AD$, such that
\begin{equation}\label{b22}
M_{n_j,i_j,w_j} \subset B_j, \ \ f_{w_j}^{n_j}(M_{n_j,i_j,w_j}^\circ)=B_{j+1}, \ \text{and} \ f_{w_N}^{n_N}(M_{n_N,i_N,w_N}^\circ)=B_{1}.
\end{equation}
Put $U:=M \setminus \bigcup_{B_i\in\mathcal{B}} \partial B_i$ and take $\widetilde{U}=U \setminus \bigcup_{j=1}^N M_{n_j,i_j,w_j}$ which are both open subsets of $M$. Applying Proposition \ref{pro22} to $\widetilde{U}$, one can obtain a countably infinite subfamily $\widetilde{\mathcal{M}}=\{M_i: i\in I\}\subseteq \mathcal{V}_\ell$, for some countable index set $I$, of disjoint closed dynamical balls such that $\lambda(\widetilde{U}\triangle \bigcup_{i \in I}M_i)=0$. Moreover, the elements of $\widetilde{\mathcal{M}}$ are contained in $\widetilde{U}$. In particular, for each $i \in I$, $M_i^\circ \cap \partial B_i=\emptyset$ and $M_i^\circ \cap M_{n_j,i_j,w_j}=\emptyset$.
Note that for each $M_i \in \widetilde{\mathcal{M}}$, there exists some $n_i \in \mathbb{N}$, $j_i \in \{1, \ldots, N\}$ and $w_i \in AD$ in such a way that $M_i=M_{n_i,j_i,w_i}$
and $f_{w_i}^{n_i}(M_i^\circ)=B_{j_i}$. For simplicity, put $h_i=f_{w_i}^{n_i}$.
Take $\widetilde{\mathcal{H}}:=\{h_{i}:i \in I\}\subset \Gamma$ corresponding to $\widetilde{\mathcal{M}}$. Clearly, $h_i(M_i^\circ)\in \mathcal{B}$.
It is easily seen that the collection $\mathcal{M}=\widetilde{\mathcal{M}}\cup \{M_{n_j,i_j,w_j}: j=1, \ldots,N\}$ together with $\mathcal{H}=\widetilde{\mathcal{H}}\cup \{ f_{w_j}^{n_j}: j=1, \ldots, N\}$ is a Markov partition for $\Gamma$ with the finite images and finite cycle properties, where $M_{n_j,i_j,w_j}$ and $f_{w_j}^{n_j}$ given by (\ref{b22}).
Indeed, suppose that $h_j(M_j^\circ)\cap M_i^\circ\neq\emptyset$. Since $h_j(M_j^\circ)\in\mathcal{B}$ and $M_i$ doesn't intersect $\bigcup_{B_i\in\mathcal{B}}\partial B_i$, one has that $M_i\subset h_j(M_j)$, that is $(\mathcal{M},\mathcal{H})$ satisfies the Markovian property.
Moreover, for each $i$,
since $\mathcal{B}$ is a cover of $M$ and $\text{diam}(M_i)$ is less than the Lebesgue number of $\mathcal{B}$, hence there exists $j\in \{1, \ldots, N\}$with $M_i \subset B_{j+1}$, furthermore, by (\ref{b22}), $f_{w_j}^{n_j}(M_{n_j,i_j,w_j}^\circ)=B_{j+1}$ and hence $f_{w_j}^{n_j}(M_{n_j,i_j,w_j})\supset M_i$. Also, $h_i(M_i^\circ)=B_j$, for some $B_j \in \mathcal{B}$ and therefore, $h_i(M_i)\supset M_{n_j,i_j,w_j}$. Thus the Markov partition $(\mathcal{M},\mathcal{H})$ satisfies FIP and FCP. Since $\mathcal{M}$ and $\mathcal{H}$ are countable, they can be reindexed with the positive integers, giving us the positive integers $b_j$, $j=1, \ldots, N$, for which $M_{b_j}:=M_{n_j,i_j,w_j}\in \mathcal{M}$, $h_{b_j}:=f_{w_j}^{n_j} \in \mathcal{H}$ and $h_{b_j}(M_{b_j}^\circ)=B_{j+1}$, for $j=1, \ldots, N-1$ and $h_{b_N}(M_{b_N}^\circ)=B_{1}$. This finishes the proof.
\end{proof}

\vspace{-.3cm}
\begin{figure}[h!]
\begin{center}
\includegraphics[scale=0.5]{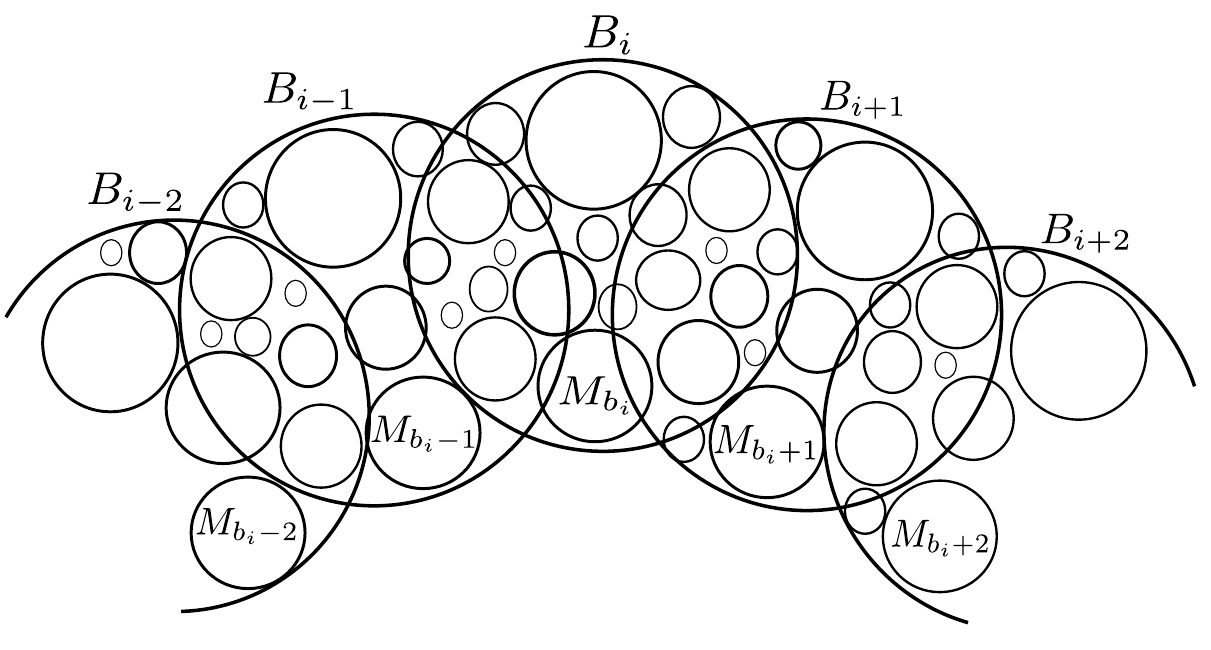}
\end{center}
\caption{ A countable Markov partition with FIP and FCP}
\label{fig:1}
\end{figure}
Theorem A is an immediate consequence of Theorem \ref{cmarkov}.
\subsection{Bounded distortion property and ergodicity of induced maps}\label{section4}
Fix any locally conformal topologically mixing semigroup $\Gamma$ of $C^{1+\alpha}$ conformal (local) diffeomorphisms on a compact connected manifold $M$.
By Theorem \ref{cmarkov}, the semigroup $\Gamma$ admits a countable Markov partition $(\mathcal{M}, \mathcal{H})=\{(M_{i}, h_i)\}_{i\in\mathbb{N}}$ with the finite images and the finite cycle properties.
In particular, $h_i(M_i^\circ)\in \mathcal{B}=\{B_1, \ldots, B_N\}$.
In the rest, we fix elements $M_{b_j}\in \mathcal{M}$ with the corresponding maps $h_{b_j}$, $j=1, \ldots, N$, such that $h_{b_j}(M_{b_j}^\circ)\in \mathcal{B}$ and they satisfy the finite cycle property. This means that for indices $b_j \in \mathbb{N}$, $j=1, \ldots, N$, we have
$$h_{b_j}(M_{b_j}^\circ)=B_{j+1}, \ \text{for} \ j=1, \ldots, N-1 \ \text{and} \ h_{b_N}(M_{b_N}^\circ)=B_{1}.$$
Let us take the induced map $F :W \to M$ given by (\ref{ind}).
We will obtain a bounded distortion formulation for the induced map $F$.
\begin{remark}\label{exact}
The following two facts hold.
\begin{enumerate}
\item [(i)] By FIP and FCP, for each $i=1, \ldots, N$, one has that
$F^{N}(M_{b_i}^\circ)\supseteq W:=\bigcup_{\ell=1}^\infty M_\ell^\circ$. In particular, $F$ is topologically mixing.
\item [(ii)]
Each $M_j \in \mathcal{M}$ is a closed dynamical $n$-ball, for some $n \in \mathbb{N}$. Define inductively
closed dynamical $(n+n_{i_1}+ \cdots + n_{i_\ell})$-balls for $\ell \in \mathbb{N}$ as follows.
$M_{i_0}:=M_j$, $h_{i_0}:=h_j$, and for any $\ell \geq 1$, let
\begin{equation}\label{cyl}
M_{i_0,\ldots,i_\ell}:=h_{i_\ell}^{-1}(M_{i_0,\ldots,i_{\ell-1}}) \ \text{with} \ M_{i_0,\ldots,i_{\ell-1}}\subset h_{i_\ell}(M_{i_\ell}),
\end{equation}
where $n_{i_j}$ is the length of $h_{i_j}$.
Hence, we have the following chain of expanding maps:
\begin{equation*}\label{chain2}
M_{i_0,\ldots,i_\ell}\xrightarrow{h_{i_\ell}} M_{i_0,\ldots,i_{\ell-1}}\cdots\xrightarrow{h_{i_{0}}}h_{i_0}(M_{i_0})
\end{equation*}
with $h_{i_0}(M_{i_0}^\circ)\in\mathcal{B}$.
In particular,
$\text{diam}(M_{i_0,\ldots,i_\ell})\leq \sigma^{n_{i_\ell}} \text{diam}(M_{i_0,\ldots,i_{\ell-1}})$.
\end{enumerate}
\end{remark}
\begin{lemma}[Bounded Distorsion]\label{lem11}
There exists $K_1>1$ such that for any dynamical ball $M_{i_0,\ldots,i_\ell}$ given by (\ref{cyl}) in Remark \ref{exact} and for each $y,z \in M_{i_0,\ldots,i_\ell}^\circ$ one has
\begin{equation*}
K_1^{-1} \leq \frac{|\text{det} (DF^j(y))|}{|\normalfont{det}(DF^j(z))|} \leq K_1, \ j=0, \ldots, \ell+1.
\end{equation*}
\end{lemma}
\begin{proof}
By property (ii) of Remark \ref{exact}, for each dynamical ball $M_{i_0,\ldots,i_\ell}$ we have
$$h_{i_0}\circ h_{i_{1}}\circ\cdots\circ h_{i_\ell}(M_{i_0,\ldots,i_\ell}^\circ)=B_{m}\in\mathcal{B}, \ \text{for \ some} \ m \in \{1, \ldots, N\}$$
and hence $F^{\ell+1}(M_{i_0,\ldots,i_\ell}^\circ)=B_m$.
In particular,
$M_{i_0,\ldots,i_\ell}^\circ=h_{i_\ell}^{-1}\circ \cdots \circ h_{i_0}^{-1}(B_{m})$.
Moreover, for any $x \in M_{i_0,\ldots,i_\ell}^\circ$ and any $1 \leq j \leq \ell$ one has that
$h_{i_{\ell -j}}\circ \cdots \circ h_{i_\ell}(x) \in h_{i_{\ell -j-1}}^{-1}\circ \cdots \circ h_{i_0}^{-1}(B_{m})$.
Also, by construction and the proof of Proposition \ref{pro1}, for each $0 \leq j \leq \ell$, there exists an admissible finite word $w_j$ of the alphabets $\{1, \ldots, k\}$ with the length $n_j$ such that $h_{i_j}=f_{w_j}^{n_j}$. Take $n=n_0+ \cdots+ n_\ell$. Then there is an admissible
finite word $w=(j_1, \ldots, j_n)$ of the alphabets $\{1, \ldots, k\}$ with the length $n$ such that $ h_{i_0}\circ h_{i_{1}}\circ\cdots\circ h_{i_\ell}=f_w^n$, where $f_w^n=f_{j_1}\circ \cdots \circ f_{j_n}$.
Finally, by construction, for each $1 \leq t \leq n$, one has that
\begin{equation*}\label{cylinder}
\text{diam}(f_{j_{n -t}}^{-1}\circ \cdots \circ f_{j_1}^{-1}(B_m)\leq \sigma \text{diam}( f_{j_{n -t-1}}^{-1}\circ \cdots \circ f_{j_1}^{-1}(B_m),
\end{equation*}
where $0 <\sigma <1$ is given by condition $(\normalfont{\textbf{C1}})$. By assumption, the mapping $g_i:=\log \mid det Df_i|_{V_i}\mid$, for all $1\leq i \leq k$, is $\alpha$-H\emph{$\ddot{o}$}lder and thus for all $z,y\in V_i$, it holds that
$\mid g_i(z)-g_i(y)\mid\leq C_0d(z,y)^\alpha$,
for some constants $C_0>0$ and $\alpha>0$ (independent of $i$). These implies that for each $z,y \in M_{i_0,\ldots,i_\ell}^\circ$ we have
$$d(f_w^{t}(z),f_w^{t}(y))\leq \sigma^{n-t} \text{diam}(\mathcal{B})\leq 2\varepsilon \sigma^{n-t},$$
where $f_w^t=f_{j_1}\circ \cdots \circ f_{j_t}$, $1\leq t\leq n$.
Therefore
\begin{align*}
\log\dfrac{\mid det (Df^t_w(z))\mid}{\mid det  (Df^t_w(y))\mid} = \sum_{k=1}^{t}\mid g_{j_{t-k +1}}(f_w^{k-1}(z))-g_{j_{t-k +1}}(f_w^{k-1}(y))\mid
\leq & \sum_{k=1}^{t}C_0 d(f^{k-1}_w(z),f^{k-1}_w(y))^\alpha \\
\leq & C_0 (2\varepsilon)^\alpha \sum_{k=1}^{t}(\sigma^{\alpha(n-t)})\\
\leq & C_0(2\varepsilon)^\alpha\sum_{k=0}^{\infty}\sigma^{k \alpha}.
\end{align*}
Take $K_1=\exp(C_0(2\varepsilon)^\alpha\sum_{k=0}^{\infty}\sigma^{k \alpha}$. Then by this choice of $K_1$, the inequality of the lemma holds.
\end{proof}
Let $h_j(M_j^\circ)=B_i$, for some $B_i \in \mathcal{B}$. If $M_k^\circ \subset B_i$ then $h_j^{-1}(M_k^\circ)\subset M_j^\circ$. We say that $h_j^{-1}|_{M_k^\circ} : M_k^\circ \to h_j^{-1}(M_k^\circ)\subset M_j^\circ$
is an {\it inverse branch} of $F$. More generally, for any $n \geq 1$, we call an inverse branch of $F^n$ a composition
$$h_{i_n}^{-1}\circ \cdots \circ h_{i_1}^{-1}|_{M_k^\circ}:M_k^\circ \to h_{i_n}^{-1}\circ \cdots \circ h_{i_1}^{-1}(M_k^\circ)\subset M_{i_n}^\circ,$$ where $(i_1, \ldots, i_n)$
is an admissible word, $M_k \in \mathcal{M}$
with $h_{i_1}^{-1}(M_k^\circ)\subset M_{i_1}^\circ$ and $h_{i_j}^{-1}\circ \cdots \circ h_{i_1}^{-1}(M_k^\circ)\subset M_{i_{j}}^\circ$, for $1\leq j \leq n$.

\begin{proposition}\label{pro222}
There are $C_0 > 0$ and a $F$-invariant absolutely continuous probability measure $\mu_F=\rho_F \lambda$ with $C_0^{-1}\leq \rho_F \leq C_0$, where $\lambda$ denotes the normalized Lebesgue measure. In particular, $\mu_F$ is ergodic.
\end{proposition}
\begin{proof}
For each $x \in W$, take $\normalfont{J}_F(x)=|det  DF(x)|$.
By Lemma \ref{lem11}, we have the following inequality: for each $y,z \in M_{i_0,\ldots,i_{\ell -1}}^\circ$,
$$K_1^{-1}\leq \frac{J_{F^\ell}(y)}{J_{F^\ell}(z)} \leq K_1.$$
By the change of variable formula, there is a
constant $K_2 > 0$ such that for any dynamical ball $M_{i_0,\ldots,i_{\ell -1}}$ given by (\ref{cyl}), with $F^{\ell}(M_{i_0,\ldots,i_{\ell -1}}^\circ)=B_m$, for some $B_m \in \mathcal{B}$, if $A_1, A_2 \subset B_m$ are two measurable subsets then
\begin{equation}\label{bdd2}
K_2^{-1}\frac{\lambda(A_1)}{\lambda(A_2)} \leq \frac{\lambda(F^{-\ell}(A_1))}{\lambda(F^{-\ell}(A_2))} \leq K_2 \frac{\lambda(A_1)}{\lambda(A_2)},
\end{equation}
where $F^{-\ell}$ is the corresponding inverse branch defined above. 
For each $\ell \in \mathbb{N}$, consider the collection $\mathcal{W}_\ell$ composed of all dynamical balls $M_{i_0,\ldots,i_{\ell -1}}$ given by (\ref{cyl}). Clearly $\mathcal{W}_\ell$
is a countable partition of $W$. First, we claim that
there exists $C_0 >0$ such that for each $\ell \in \mathbb{N}$ and any measurable subset $A$ of $W$ one has that
\begin{equation}\label{bdd3}
C_0^{-1}\lambda(A) \leq \lambda(F^{-\ell}(A)) \leq C_0 \lambda(A).
\end{equation}
Actually, for any inverse branch as above, set $A_2=B_m$ and $A_1=A \cap M_k$, for each $M_k \subset B_m$. Then
$\lambda(F^{-\ell} (A))$ is the sum of the terms $\lambda(F^{-\ell}(A \cap M_k))$, with $M_k \subset B_m$ and $B_m \in \mathcal{B}$, over all
inverse branches of $F^\ell$. 
The claim now follows from (\ref{bdd2}). As a classical method, using (\ref{bdd3}) we get that every accumulation point $\mu_F$ of the sequence
\begin{equation*}\label{mea}
\mu_n =\frac{1}{n}\sum_{i=0}^{n-1}F^i_{\ast}\lambda
\end{equation*}
is a $F$-invariant probability absolutely continuous with respect to $\lambda$, with density
$\rho_F$ bounded from zero and from infinity. Indeed, by (\ref{bdd3}), for every $n$, the density $\rho_n = d \mu_n/d \lambda$ satisfies $C_0^{-1} \leq \rho_n \leq C_0$
and the same holds for the density of the accumulation point $\mu_F$ (see Theorem 1.3 in \cite{Man}). Additionally, by part (i) of Remark \ref{exact} and absolute continuity of $\mu_F$, the measure $\mu_F$ is ergodic, see proof of Theorem 1 in \cite{Ya} for more details.
\end{proof}
\section{Markov towers and Thermodynamics of two-sided countable shifts}\label{section5}
This section is devoted to exhibiting an inducing
scheme of hyperbolic type satisfying conditions (\textbf{H1})-(\textbf{H4}) for each locally conformal expanding $C^{1+\alpha}$ action with topological mixing  property. By utilization of the inducing
scheme we model semigroup actions by towers over the subshift of countable type on the
space of two-sided infinite sequences. Moreover we  establish a thermodynamic
formalism for this kind of countable subshifts.
\subsection{Markov towers for locally conformal semigroup actions}\label{section5}
Let $\Gamma$ be a locally conformal expanding $C^{1+\alpha}$ action with topological mixing  property a compact connected manifold $M$ and take
the Markov partition $(\mathcal{M}, \mathcal{H})=\{(M_{i}, h_i)\}_{i\in\mathbb{N}}$ with FIP and FCP given by Theorem \ref{cmarkov}.
Let $\Sigma_{\mathcal{A}}^{(+)}$ and $\Sigma_{\mathcal{A}}$ be one-sided and two-sided topological Markov chains
defined by
(\ref{chain3})
with the left shift maps, $\sigma_{\mathcal{A}}^{(+)}$.
The aim is to show that $\Gamma$ exhibits an inducing scheme
of hyperbolic type satisfying conditions (\textbf{H1})-(\textbf{H4}).
The next proposition verifies conditions (\textbf{H1}) and (\textbf{H2}).
\begin{proposition}\label{pros}
The following statements hold:
\begin{enumerate}
\item [(1)]
For each $M_i \in \mathcal{M}$, $h_i |_{M_i^\circ}$ can be extended to a homeomorphism of a neighborhood of $M_i$.
\item [(2)] For every bi-infinite sequence $\textbf{i}=(\ldots, i_{-1}, i_0, i_1, \ldots)\in\Sigma_{\mathcal{A}}$ there exists a unique
sequence $\underline{x}=\underline{x}(\textbf{i})=(x_n)_{n \in \mathbb{Z}}$ such that
\begin{enumerate}
\item [$(a)$] $x_n \in M_{i_n}$ and $h_{i_n}( x_n)=x_{n+1}$;
\item [$(b)$] if for some sequences $\underline{x}=\underline{x}(\textbf{i})$ and $\underline{y}=\underline{y}(\textbf{j})$,
one has $ x_n = y_n $ for all $n \leq 0$ then $(\ldots, i_{-1}, i_0)=(\ldots, j_{-1}, j_0)$.
\end{enumerate}
\end{enumerate}
\end{proposition}
\begin{proof}
The proof follows directly from Theorem \ref{cmarkov}.
Indeed, let $|h_i|$ be the length of the admissible finite word $w_i$, $i \in \mathbb{N}$, so that $h_i=f_{w_i}^{|h_i|}$.
By definition, it is clear that $h_i |_{M_i^\circ}$ can be extended to a homeomorphism of a neighborhood of $M_i$ and statement (1) is verified. Let $$\textbf{i}:=(\ldots, i_{-n}, \ldots, i_{-1}, i_0, i_1, \ldots, i_n,\ldots)\in \Sigma_{\mathcal{A}}$$ and take the natural projections
$$ q_n(\textbf{i}):=q_n(\ldots, i_{-n}, \ldots, i_{-1}, i_0, i_1, \ldots, i_n,\ldots)=(i_0, i_1, \ldots, i_n), \quad n\geq 0.$$
For the admissible word $q_n(\textbf{i})=(i_0, i_1, \ldots, i_n)$, take the cylinder set
\begin{equation}\label{cy}
C_{i_0,\ldots,i_{n}}=M_{i_0}\cap h_{i_0}^{-1}(M_{i_1})\cap \cdots \cap h_{i_0}^{-1} \circ h_{i_1}^{-1}\circ \cdots \circ h_{i_{n-1}}^{-1}(M_{i_n}).
\end{equation}
Then, by construction, for each $n \in \mathbb{N}$, $C_{i_0,\ldots,i_{n}}\subset C_{i_0,\ldots,i_{n-1}}$. Also,
$$\text{diam}(C_{i_0,\ldots,i_n})\leq \sigma^{|h_{i_n}|} \text{diam}(C_{i_0,\ldots,i_{n-1}}).$$
Thus, the intersection $\bigcap_{n\geq 0}C_{i_0,\ldots,i_{n}}$ is a single point, denoted by $x_0$.
In particular, $x_0\in M_{i_0}$, $h_{i_0}(x_0)\in M_{i_1}$, and $h_{i_{n-1}}\circ \cdots \circ h_{i_0}(x_0)\in M_{i_n}$, for each $n \in \mathbb{N}$. Call $\textbf{i}=(\ldots, i_{-1}, i_0, i_1, \ldots)$ is called the \emph{itinerary} of $x$. Take $x_1=h_{i_0}(x_0)$ and $x_n=h_{i_{n-1}}\circ \cdots \circ h_{i_0}(x_0)$, for each $n \in \mathbb{N}$.
Let $x_{-1}=h_{i_{-1}}^{-1}(x_0)$ and inductively, we set $x_{-n}=h_{i_{-n}}^{-1}(x_{-n+1})$.
Then, for every bi-infinite sequence $\textbf{i}=(\ldots, i_{-1}, i_0, i_1, \ldots)\in\Sigma_{\mathcal{A}}$, we can assign a unique
sequence $\underline{x}=\underline{x}(\textbf{i})=(x_n)_{n \in \mathbb{Z}}$ as above. By construction, this sequence satisfies part ($a$) of statement (2).
Part($b$) is an immediate consequence of our construction.
\end{proof}
\begin{proposition}\label{pi}
The map $\pi$ given by (\ref{code}) has the following properties:
\begin{enumerate}
\item [(1)] $\pi$ is well defined, continuous and for all $\textbf{i} \in \Sigma_{\mathcal{A}}$ one has
$\pi \circ \sigma_{\mathcal{A}}(\textbf{i})=h_{i_0} \circ \pi(\textbf{i})$,
where $i_0$ is the zeroth digit in $\textbf{i}$,
\item [$(2)$] $\pi$ is one-to-one on $\check{\Sigma}$ and $\pi(\check{\Sigma})\subset W$, where $\check{\Sigma}$ is defined by (\ref{hat}) and $W=\bigcup_{M_i \in \mathcal{M}}M_i^\circ$,
\item [$(3)$] if $\pi(\textbf{i})=\pi(\textbf{j})$ for some $\textbf{i},\textbf{j}\in \check{\Sigma}$ then $i_n=j_n$ for all $n \geq 0$.
\end{enumerate}
\end{proposition}
\begin{proof}
Part (1) is clear by the definition of $\pi$. To prove (2), first note that $\pi$ is one-to-one on $\check{\Sigma}$.
Take
$\Delta := \pi^{-1}(\partial \mathcal{M})$, where $\partial \mathcal{M}:= \bigcup_{M_i \in \mathcal{M}}M_i \setminus M_i^\circ$.
By definition of $\check{\Sigma}$, for any $\textbf{i} \in \Sigma_{\mathcal{A}}\setminus \check{\Sigma}$ there exist $k \in \mathbb{Z}$ such that $\pi \circ \sigma_{\mathcal{A}}^k(\textbf{i})\in M_{i_k}\setminus M_{i_k}^\circ$.
Hence, $\check{\Sigma}=\Sigma_{\mathcal{A}} \setminus \bigcup_{k\geq 0}\sigma_\mathcal{A}^{-k}(\Delta)$.
By this fact and property (1), we get $\sigma_{\mathcal{A}}(\check{\Sigma})\subset \check{\Sigma}$ and $\pi(\check{\Sigma})\subset W$ and so, property (2) is verified.
Property (3) follows immediately from the definitions of $\pi$ and $\check{\Sigma}$
and Proposition \ref{pros}.
\end{proof}
We set
\begin{equation}\label{w1}
\widetilde{W}:=\pi(\check{\Sigma}).
\end{equation}
Condition (\textbf{H3}) can be followed by the next result.
\begin{proposition}\label{proin}
Let $\mu$ be an ergodic $\sigma_{\mathcal{A}}$-invariant measure which gives positive weight to any open subset. Then $\mu (\Sigma_{\mathcal{A}}\setminus \check{\Sigma})=0$
and the map $\pi$ given by (\ref{code}) is a continuous bi-measurable bijection map from $\check{\Sigma}$
to a (Lebesgue) full subset of $M$.
\end{proposition}
\begin{proof}
Let $\mu$ be an ergodic $\sigma_{\mathcal{A}}$-invariant measure which gives positive weight to any open subset and consider the map $\pi$ given by (\ref{code}).
By the proof of Proposition \ref{pi}, for $\Delta$ give above, \linebreak$\check{\Sigma}=\Sigma_{\mathcal{A}} \setminus \bigcup_{k\geq 0}\sigma_\mathcal{A}^{-k}(\Delta)$.
Thus, $\pi : \check{\Sigma} \to \widetilde{W}$
is bi-measurable, injective and also surjective to full Lebesgue measure $\widetilde{W}$. has full Lebesgue measure.
Also, Proposition \ref{pi} ensures that $\sigma_\mathcal{A}(\Delta)\subset \Delta$, i.e. $\Delta$ is forward invariant.
Since the measure $\mu$ is $\sigma_{\mathcal{A}}$-invariant, $\mu(\sigma_{\mathcal{A}}^n(\Delta))=\mu(\sigma_{\mathcal{A}}(\Delta))$. Hence, the set $\bigcap_{n\geq 0}\sigma_{\mathcal{A}}^n(\Delta)$ has $\mu$-measure 0 or 1 as $\mu$ is $\sigma_{\mathcal{A}}$-ergodic.
Since its complement is a nonempty open set and has positive measure by the fact that $\mu$ gives positive weight to
any open subset, one gets $\mu(\Delta)=0$. This completes the proof.
\end{proof}
By part (i) of Remark \ref{exact}, the induced map $F$ is topologically mixing, thus we get the next result. It verifies condition (\textbf{H4}).
\begin{proposition}\label{period}
The induced map $F$ has at least one periodic point in $W$.
\end{proposition}
Theorem B is now an immediate corollary of Propositions \ref{pros}, \ref{proin} and \ref{period}.
\begin{remark}
Proposition \ref{proin} ensures that every Gibbs measure of $\sigma_\mathcal{A}$ is
supported on $\check{\Sigma}$ and its projection by $\pi$ is thus supported on $\widetilde{W} \subset W$ and is $F$-invariant.
Proposition \ref{period} guarantees the finiteness of the pressure function.
\end{remark}
\subsection{Thermodynamics of two-sided countable shifts}\label{section6}
Locally conformal semigroup actions with inducing schemes are
modeled by towers over the subshift of countable type on the space of two-sided infinite
sequences. This motivate us to establish a thermodynamic formalism for this kind of countable subshifts.
Thermodynamic formalism for subshifts on the space of one-sided infinite
sequences over a countable alphabet provided by Sarig \cite{Sa1, Sa2}.
In this section, we extend this result to countable subshifts on the space of two-sided infinite sequences.
The main result in this section is Theorem \ref{mt} which
provides conditions on the potential functions that guarantee existence and uniqueness
of equilibrium measures and describes their ergodic properties.
We use these results in Subsections \ref{Section: 5.1} and \ref{section8} to establish the thermodynamic formalism for
locally conformal semigroup actions with inducing schemes of hyperbolic type.
Given a locally conformal semigroup $\Gamma$, take the Markov partition $(\mathcal{M}, \mathcal{H})=\{(M_{i}, h_i)\}_{i\in\mathbb{N}}$ having the finite images and the finite cycle properties.
Consider the elements $M_{b_j}\in \mathcal{M}$ with the corresponding maps $h_{b_j}$, $j=1, \ldots, N$, such  that $h_{b_j}(M_{b_j}^\circ)\in \mathcal{B}$ and they satisfy the finite cycle property. That is for indices $b_j \in \mathbb{N}$, $j=1, \ldots, N$, we have
$h_{b_j}(M_{b_j}^\circ)=B_{j+1}$, for $j=1, \ldots, N-1$ and  $h_{b_N}(M_{b_N}^\circ)=B_{1}$.
Let us take the induced map $F :W \to M$ given by (\ref{ind}) and the two sided Markov chain $(\Sigma_{\mathcal{A}},\sigma_{\mathcal{A}})$. 

We recall preliminaries of thermodynamic formalism for countable subshifts.
For a potential $\varphi:\Sigma_{\mathcal{A}} \to \mathbb{R}$ we denote the $n$-th \emph{variation} of $\varphi$.
by
\begin{equation*}\label{pot}
\text{Var}_n(\varphi):=\sup\{|\varphi(\textbf{i})-\varphi(\textbf{j})|: \textbf{i},\textbf{j} \in \Sigma_{\mathcal{A}}, \ i_\ell=j_\ell, \ -n+1 \leq \ell \leq n-1\}
\end{equation*}
\begin{itemize}
  \item $\varphi$ has  \emph{summable variations} if $ \sum_{n=1}^\infty \text{Var}_n(\varphi)< \infty$;
  \item $\varphi$ has  \emph{strongly summable variations} if $\sum_{n=1}^\infty n \text{Var}_n(\varphi)< \infty$;
  \item $\varphi$ is \emph{locally H\emph{$\ddot{\emph{o}}$}lder continuous} if there exist $C > 0$ and $0 < \theta < 1$
  such that for all $n \geq 1$, $\text{Var}_n(\varphi)\leq C \theta^n$. Clearly, locally H\emph{$\ddot{\emph{o}}$}lder continuity guarantees summable variation.
\end{itemize}
For $\ell \in \mathbb{N}$, let
\begin{equation*}\label{z}
 Z_n(\varphi, \ell):=\sum_{\sigma_{\mathcal{A}}^n(\textbf{i})=\textbf{i}, i_0=\ell}\exp (\varphi_n(\textbf{i})),
\end{equation*}
where $\varphi_{n}(\textbf{i})=\sum_{k=0}^{n-1}\varphi \circ \sigma_{\mathcal{A}}^{k}(\textbf{i})$.
The \emph{Gurevic-Sarig} pressure of $\varphi$ is the number
$$\mathcal{P}(\varphi):=\lim_{n \to \infty}\frac{1}{n} \log Z_n(\varphi, \ell).$$
This notion introduced by Gurevic in \cite{G1, G2} which is a generalization of the notion of topological entropy
$h_G(\sigma_{\mathcal{A}})$ for countable Markov chains, so that $\mathcal{P}(0) = h_G(\sigma_{\mathcal{A}})$.
Since the Gurevich pressure only depends on the positive side of the sequences,
$ \mathcal{P}(\varphi)$ exists whenever $\sum_{n=1}^\infty n \text{Var}_n(\varphi)< \infty$
 and it is independent of $\ell \in \mathbb{N}$ by Theorem 1 of \cite{Sa1}. A \emph{Gibbs\ measure} for $\varphi$ is an invariant probability measure $\mu_\varphi$ such that for some constant $B$ and every cylinder $[a_{0},\ldots , a_{n-1}]$,
$$\dfrac{1}{B}\leq\dfrac{\mu_\varphi([a_{0},\ldots,a_{n-1}])}{e^{\varphi_{n}(\textbf{i})-n \mathcal{P}(\varphi)}}\leq B,$$
for all $ \textbf{i}\in[a_{0},\ldots,a_{n-1}]$, where
$\varphi_{n}(\textbf{i})=\sum_{k=0}^{n-1}\varphi \circ \sigma_{\mathcal{A}}^{k}(\textbf{i})$. Let $\mathcal{M}(\sigma_{\mathcal{A}})$ denote the set of $\sigma_{\mathcal{A}}$-invariant Borel probabilities on $\Sigma_{\mathcal{A}}$ and
\begin{equation*}
\mathcal{M}_\varphi(\sigma_{\mathcal{A}}):=\{\nu \in \mathcal{M}(\sigma_{\mathcal{A}}): \int_{\Sigma_{\mathcal{A}}}\varphi d\nu > -\infty \}.
\end{equation*}
An \emph{equilibrium measure} is a $\sigma_{\mathcal{A}}$-invariant Borel probability
measure $\nu_\varphi$ for which the following holds:
\begin{equation*}\label{pr}
h_{\nu_\varphi}(\sigma_{\mathcal{A}}) + \int \varphi d\nu_\varphi=\sup_{\nu \in \mathcal{M}_\varphi(\sigma_{\mathcal{A}}) } \{h_\nu(\sigma_{\mathcal{A}}) + \int \varphi d\nu\}< \infty.
\end{equation*}
In the following, the thermodynamic formalism for countable shift $\sigma_{\mathcal{A}}$ on the space $\Sigma_{\mathcal{A}}$ of two-sided infinite sequences is established.
\begin{theorem}\label{mt}
Assume $\sup_{\textbf{i}\in \Sigma_{\mathcal{A}}} \varphi(\textbf{i}) < \infty$ and $\varphi$ has strongly summable variations. Then:
\begin{enumerate}
  \item [$(1)$] The variational principle for $\varphi$ holds, i.e.,
$\mathcal{P}(\varphi)=\sup_{\nu \in \mathcal{M}_\varphi(\sigma_{\mathcal{A}}) } \{h_\nu(\sigma_{\mathcal{A}}) + \int \varphi d\nu\}.$
  \item [$(2)$] If $\mathcal{P}(\varphi)<\infty$ then there exists a unique $\sigma_{\mathcal{A}}$-invariant ergodic Gibbs measure $\nu_\varphi$ for $\varphi$.
  \item [$(3)$] If, furthermore, $h_{\nu_\varphi} (\sigma_{\mathcal{A}}) < \infty$, then $\nu_\varphi \in \mathcal{M}_\varphi(\sigma_{\mathcal{A}})$ and it is the unique equilibrium measure for $\varphi$.
\end{enumerate}
\end{theorem}
The proof of the theorem is based on the technique used in \cite{PSZ2} and will require the lemmas stated below. We recall that the one-sided topological Markov chain $ (\Sigma_{\mathcal{A}}^{+},\sigma_{\mathcal{A}}^+)$ has the big images and preimages (BIP) property if
$$\exists\ b_{1},\ldots,b_{N}\,\text{such that}\,\forall a\in \mathbb{N} \ \exists\ i,j,\  t_{b_{i}a}t_{ab_{j}}=1.$$
For induced map $F$ given by (\ref{ind}), we say that $\varphi$ is a \emph{piecewise H\emph{$\ddot{\emph{o}}$}lder continuous potential} of $F$ if the restriction of $\varphi$ to
the interior of any element of the partition $\mathcal{M}$ is H\emph{$\ddot{\emph{o}}$}lder continuous, i.e. for all $x, y$ in the interior of the same element of $\mathcal{M}$,
$|\varphi(x)-\varphi(y)|\leq K d(x,y)^\alpha$,
for some $\alpha > 0$, $K <\infty$.
\begin{proposition}\label{l}
The one-sided topological Markov chain $(\Sigma_{\mathcal{A}}^{+}, \sigma_{\mathcal{A}}^+)$ has the BIP property.
\end{proposition}
\begin{proof}
By Theorem \ref{cmarkov} the semigroup $\Gamma$ exhibits a Markov partition $(\mathcal{M},\mathcal{H})$ with the finite images and the finite cycle properties and
 let $\mathcal{A}=(t_{ij})_{\mathbb{N} \times \mathbb{N}}$ be the corresponding transition matrix given by (\ref{h}) such that $t_{ij}=1$ if $M_i^\circ \cap h_i^{-1}(M_j^\circ) \neq \emptyset$.
Since $(\mathcal{M},\mathcal{H})$ has the finite cycle property, $\mathcal{B}=\{h_i(M_i^\circ): i \in \mathbb{N}\}$ consists of finitely many
open subsets $\{B_1, \ldots, B_N \}$ of $M$.
By FCP, there exist integer numbers $b_\ell$, $\ell=1, \ldots, N-1$, so that $M_{b_\ell}\subset B_\ell$, $h_{b_\ell}(M_{b_\ell}^\circ)=B_{\ell +1}$, $M_{b_N}\subset B_N$ and $h_{b_N}(M_{b_N}^\circ)=B_{1}$. Given $a \in \mathbb{N}$, By the construction of the Markov partition, there exists $i \in \{1, \ldots, N\}$ such that $M_a \subset B_i$. Hence, $h_{b_{i-1}}(M_{b_{i-1}}^\circ)\supset B_i \supset M_a$ which ensures that
$t_{b_{i-1}a}=1$. Also, $h_a(M_a^\circ)=B_j$ for some $j \in \{1, \ldots, N\}$. Then by the finite cycle property, $h_a(M_a^\circ)\supset M_{b_j}$, and hence $t_{a b_j}=1$.
\end{proof}
It is not hard to see that there exists a one to one correspondence between the cylinder sets $C_{i_0,\ldots,i_\ell}$ given by (\ref{cy}) and the cylinder sets of $\Sigma_{\mathcal{A}}^{+}$.
 Hence, according to statement (i) of Remark \ref{exact}, we get the next result.
\begin{proposition}\label{k}
One-sided topological Markov shift $ (\Sigma_{\mathcal{A}}^{+},\sigma_{\mathcal{A}}^+)$ is topologically mixing.
\end{proposition}
Consider a potential $\varphi^\ast : \Sigma_{\mathcal{A}}^{+} \to \mathbb{R}$ and define its $n$-variations $\text{Var}_n(\varphi^\ast)$ and its
Gurevich pressure $\mathcal{P}(\varphi^\ast)$ and as well as Gibbs and equilibrium measures as above but using one-sided
sequences. By Propositions \ref{k} and \ref{l}, the one-sided topological Markov shift $(\Sigma_{\mathcal{A}}^{+},\sigma_{\mathcal{A}}^+)$ is topologically mixing and has BIP.
Hence, one can get the next result which is a restating of the main results in \cite{Sa1,Sa2} in our context.
\begin{lemma}\label{2}
Assume $\sup_{\textbf{i}\in \Sigma_{\mathcal{A}}^+} \varphi^\ast(\textbf{i}) < \infty$ and $\sum_{n=2}^\infty \text{Var}_n(\varphi^\ast)< \infty$.
Then
\begin{enumerate}
  \item [$(1)$] The variational principle for $\varphi^\ast$ holds, i.e.,
 $\mathcal{P}(\varphi^\ast) =\sup_{\nu \in \mathcal{M}_{\varphi^\ast}(\sigma_{\mathcal{A}}^+) } \{h_\nu(\sigma_{\mathcal{A}}^+) + \int \varphi^\ast d\nu\}.$
  \item [$(2)$] If $\mathcal{P}(\varphi^\ast)< \infty$ and $\sum_{n=1}^\infty \text{Var}_n(\varphi^\ast)< \infty$, then there exists a unique $\sigma_{\mathcal{A}}^+$-invariant ergodic Gibbs measure $\nu_{\varphi^\ast}$ for $\varphi^\ast$.
  \item [$(3)$] If, furthermore, $h_{\nu_{\varphi^\ast}} (\sigma_{\mathcal{A}}^+) < \infty$, then $\nu_{\varphi^\ast}\in \mathcal{M}_{\varphi^\ast}(\sigma_{\mathcal{A}}^+)$ and is the unique equilibrium measure for $\varphi^\ast$.
\end{enumerate}
\end{lemma}
We extend these results to the subshift of countable type on the space of two–sided sequences. Take the natural projection $\pi_+:\Sigma_{\mathcal{A}} \to \Sigma_{\mathcal{A}}^+$ defined by
\begin{equation}\label{p+}
 \pi_+(\ldots, i_{-1}, i_0, i_1, \ldots):=(i_0, i_1, \ldots).
\end{equation}
\begin{lemma}\label{coh}
Assume $\varphi : \Sigma_{\mathcal{A}} \to \mathbb{R}$ has strongly summable variations. Then there
exists a bounded real-valued function $u : \Sigma_{\mathcal{A}} \to \mathbb{R}$ such that the function defined by
 $ \psi(\textbf{i}) := \varphi(\textbf{i}) + u(\sigma_{\mathcal{A}}(\textbf{i})) - u(\textbf{i})$
satisfies:
\begin{enumerate}
  \item [$(1)$] $\psi(\textbf{i}) = \psi(\textbf{j})$ whenever $i_n = j_n$ for all $n \geq 0$,
  \item [$(2)$] $\psi$ has summable variations, i.e., $\sum_{n=1}^\infty \text{Var}_n(\psi)< \infty$,
  \item [$(3)$] if $\varphi$ is locally H\emph{$\ddot{\emph{o}}$}lder continuous, then so is $\psi$,
  \item [$(4)$] $\mathcal{P}(\psi) = \mathcal{P}(\varphi)$.
\end{enumerate}
\end{lemma}
\begin{proof}
The proof follows directly from \cite{D} and \cite[Lemma~3.3]{PSZ2}.
Take a sequence $(r_k)_{k=-\infty}^{-1}$, with $r_k \in \mathbb{N}$ so that $t_{r_k r_{k+1}} = 1$ for each $k \leq -2$.
 By BIP, for each $\textbf{i}^+=(i_0, i_1, \cdots) \in \Sigma_{\mathcal{A}}^+$, there exists a finite
admissible word $w$ such that the concatenation $(\ldots, r_{-2}, r_{-1},w, i_0, i_1, \ldots)$ belongs to $\Sigma_{\mathcal{A}}$.
Define $r : \Sigma_{\mathcal{A}} \to \Sigma_{\mathcal{A}}$ by $r(\textbf{i})=(\ldots, r_{-2}, r_{-1},w, i_0, i_1, \ldots)$, where
$(i_0, i_1, \ldots)=\pi_+(\textbf{i})$ as defined in (\ref{p+}). Note that the finite word $w$ only depends on $r_{-1}$ and $i_0$. Let
\begin{equation*}
  u(\textbf{i}):= \sum_{j=0}^{\infty}\varphi(\sigma_{\mathcal{A}}^j(\textbf{i}))-\varphi(\sigma_{\mathcal{A}}^j(r(\textbf{i}))).
\end{equation*}
Since
  $|\varphi(\sigma_{\mathcal{A}}^j(\textbf{i}))-\varphi(\sigma_{\mathcal{A}}^j(r(\textbf{i})))| \leq Var_{j+1}(\varphi),$
the function $u$ is well defined and uniformly bounded. Applying the strong summability of $\varphi$ and the estimation in \cite[Lemma~3.3]{PSZ2}, we get
 $ \sum_{n \geq 1}Var_n(u)< \infty.$
To prove the first statement, note that
\begin{equation*}
  \psi(\textbf{i})=\sum_{j=0}^{\infty}\varphi(\sigma_{\mathcal{A}}^j(r(\textbf{i})))-\varphi(\sigma_{\mathcal{A}}^j(r(\sigma_{\mathcal{A}}(\textbf{i}))))
\end{equation*}
only depends on $i_j$ with $j \geq 0$. The other statements follow directly from \cite[Lemma.~3.3]{PSZ2}.
\end{proof}
Now, using Lemmas \ref{2} and \ref{coh} and the approach used in \cite[Theorem.~3.1]{PSZ2}, we conclude the proof of Theorem \ref{mt}.
\section{Thermodynamic formalism of the induced skew product}\label{sec7}
\subsection{Random Markov chains and liftable measures}\label{Section: 5.1}
Here, we show that the subset of liftable measures of induced skew products is nonempty.
To do that we will establish a bridge between topological Markov chains and locally conformal semigroup
actions.
We use the notion of random Markov fibred system introduced in \cite{DKS}
and define a random Markov chain in our setting. It allows us to obtain a convenient Markov measure.
The rest of our argument is based on the link between
the dynamics of the semigroup action and the induced skew product.

Given a countable index set $I$ with the discrete topology, let $I^{\mathbb{N}}$ be
the space of sequences in $I$. Assume, $\mathcal{C}$ is the Borel $\sigma$-algebra with respect to the product topology
and consider $\sigma_I:I^{\mathbb{N}} \to I^{\mathbb{N}}$, $(i_1, i_2, \ldots)\mapsto (i_2, i_3, \ldots)$ the left shift.
Assume $(\Theta, \mathcal{F}, \mu)$ is a standard probability space with a probability-preserving invertible transformation $\theta: \Theta \to \Theta$
and let
\begin{equation*}
  S: I^{\mathbb{N}} \times \Theta \to I^{\mathbb{N}} \times \Theta, \ S(\textbf{i},x):=(\sigma_I(\textbf{i}),\theta(x)).
\end{equation*}
A system of random Markov measures is defined as follows.
Let $\mathbb{P}$ be the set of probability vectors and take $\mathbb{A}$ the set of transition matrices, that is
\begin{center}
$\mathbb{P}:= \Big\{(p_i)_{i \in I} \in [0,1]^I: \sum_{i \in I} p_i=1\Big\}$\\
$\mathbb{A}:=\Big\{(\pi_{ij})_{i,j \in I} \in [0,1]^{I\times I}: \sum_{j \in I} \pi_{ij}=1 \ \text{for all} \ i\in I\Big\}$
\end{center}
Note that each pair $((p_i),(\pi_{ij}))\in \mathbb{P} \times \mathbb{A}$ defines a Markov measure $m_{((p_i),(\pi_{ij}))}$ on $(I^{\mathbb{N}},\mathcal{C})$ as follows. For $a=(i_1,\ldots,i_n)\in I^n$ and $[a]:=\{((j_{k})_{k\in\mathbb{N}},x)\in I^{\mathbb{N}} \times \Theta:j_{k}=i_{k} \ \text{for} \ k=1,\ldots,n\}$, let
\begin{equation*}
m_{((p_i),(\pi_{ij}))}([a]):=p_{i_1}\prod_{k=1}^{n-1}\pi_{i_k i_{k+1}}.
\end{equation*}
Given a measurable map $\Theta \to \mathbb{P} \times \mathbb{A}, \ x \mapsto ((p_i^x),(\pi_{ij}^x))$, a family of Markov
measures $\{m^x: x \in \Theta\}$ is defined by
\begin{equation*}
m^x_{((p_i^x),(\pi_{ij}^x))}([a]):=p_{i_1}^{x}\prod_{k=1}^{n-1}\pi^{\theta^{k-1}(x)}_{i_k i_{k+1}},
\end{equation*}
where $a=(i_1,\ldots,i_n)\in I^n$. If $S_{x}$ is non-singular with respect to $m^x$ for a.e. $x \in \Theta$, we call the system $(I^{\mathbb{N}}\times \Theta, \mathcal{C}\otimes \mathcal{F}, ((p_i^x),(\pi_{ij}^x)), S)$  a \emph{random Markov chain}. The next result ensures the existence of a $S$-invariant measure.
\begin{proposition} \cite[Proposition 5.1]{DKS} \label{propo1}
Let $(I^{\mathbb{N}}\times \Theta, \mathcal{C}\otimes \mathcal{F}, ((p_i^x),(\pi_{ij}^x)), S)$ be a random Markov chain with the following properties:
\begin{enumerate}
  \item [$(a)$] there exists a family $\{C_x\}_{x \in \Theta}$, $C_x > 1$, such that for $\mu$-a.e. $x \in \Theta$ and $i, j, k \in I$ with $\pi_{ij}^x$, $\pi_{ik}^x$, $p_j^{\theta(x)}$, $p_k^{\theta(x)}>0$,
\begin{equation*}
  C_{\theta(x)}^{-1}<\frac{\pi_{ij}^x}{\pi_{ik}^x}.\frac{p_k^{\theta(x)}}{p_j^{\theta(x)}}   < C_{\theta(x)};
\end{equation*}
  \item [$(b)$] there exists a family $\{\eta_x\}_{x \in \Theta}$, $\eta_x > 0$, such that for $\mu$-a.e. $x \in \Theta$ and $i \in I$ with
$p_i^x> 0$,
  \begin{equation*}
    \sum_{j: \pi_{ij}^x>0}p_j^{\theta(x)}> \eta_{\theta(x)}.
  \end{equation*}
\end{enumerate}
Then there exists a measurable map $x \mapsto (\mathbb{P}_{i}^x)\in \mathbb{P}$ such that the measure $d\mathbb{P}^x d \mu$ on $(I^{\mathbb{N}}\times \Theta, \mathcal{C}\otimes \mathcal{F})$
defined by $\mathbb{P}^x=m^x_{((\mathbb{P}_{i}^x),(\pi_{ij}^x))}$ is $S$-invariant.
\end{proposition}
For the locally conformal group $\Gamma$ with inducing scheme $(\mathcal{M}, \tau)$, we will define a random Markov chain as the following and show that it satisfies conditions ($a$) and ($b$)
 of the previous proposition. So, we provide a suitable Markov measure. Take the countable Markov partition $(\mathcal{M}, \mathcal{H})=\{(M_{i}, h_i)\}_{i\in\mathbb{N}}$ with the FIP and the FCP.
Let $\pi$ be the coding map given by (\ref{code}), $\check{\Sigma}\subset \Sigma_\mathcal{A}$ defined by (\ref{hat}) and take $\widetilde{W}$ from (\ref{w1}) with $\widetilde{W}=\pi(\check{\Sigma})$.
Note that, by Proposition \ref{pi}, $\pi$ is one to one on $\check{\Sigma}$ and $F$ is one to one on $\widetilde{W}$.
Take the projection $\pi_+:\Sigma_\mathcal{A}\to \Sigma_{\mathcal{A}^+}$ given by (\ref{p+}). Now, for each $x\in \widetilde{W}$, let $(p_i^{x})_{i \in\mathbb{N}}$ be a non-negative Borel measurable partition of unity defined by
\begin{equation}\label{h1}
p_i^{x}:= \left\{
 \begin{array}{rl}
  1 & \text{if}\ x\in M_i^\circ\\
  0 & \text{otherwise},
 \end{array}\right.
\end{equation}
and let $(\pi^x_{ij})_{i,j\in\mathbb{N}}$ be a right stochastic matrix (i.e. $\forall i\in\mathbb{N}, \sum_{j\in\mathbb{N}}\pi^x_{ij}=1$) defined by
\begin{equation}\label{h11}
\pi^x_{ij}:= \left\{
 \begin{array}{rl}
  1 & \text{if}\ x\in M_i^\circ \ \text{and} \ h_i(x)\in M_j^\circ\\
  0 & \text{otherwise}.
 \end{array}\right.
\end{equation}
For each finite word $\zeta=(i_1,\ldots,i_n)\in\mathbb{N}^{n}$, any $x\in \widetilde{W}$ and the associated cylinder set
$[\zeta^x]=\{((j_k)_{k\in \mathbb{N}},x): j_k \in\mathbb{N}, \ \text{and} \ j_k=i_k, 1\leq k\leq n\}$, we define a family of probability measures $\{P^x; x\in \widetilde{W}\}$  by
\begin{equation*}
P^x([\zeta^x]):=p_{i_1}^{x}\prod_{k=1}^{n-1}\pi^{F^{k-1}(x)}_{i_k i_{k+1}}.
\end{equation*}
Denote the set of all $F$-invariant ergodic Borel probability measures by $\mathcal{M}(F)$.
Note that Proposition \ref{pro222} guarantees the existence of an $F$-invariant absolutely continuous probability measure equivalent to Lebesgue.
Let $\mathcal{B}_{\check{\Sigma}_+}$ and $\mathcal{B}_{\widetilde{W}}$ be the Borel $\sigma$-algebras on $\check{\Sigma}_+$ and $\widetilde{W}$, respectively.
Given any measure $\mu_F \in \mathcal{M}(F)$, take a measure
$dP^xd\mu_F$
on $\check{\Sigma}_+ \times \widetilde{W}$.
We take $Y^+:=\check{\Sigma}_+ \times \widetilde{W}$ and define a mapping $T^+:Y^+ \to Y^+$ by
$ T^+(\textbf{i},x)=(\sigma_{\mathcal{A}^+}(\textbf{i}),F(x))$.
By definitions, $\check{\Sigma}_+$ is $\sigma_{\mathcal{A}^+}$ invariant and
$\widetilde{W}$ is $F$-invariant, hence, $T^+(Y^+)\subset Y^+$.
 Note that the mapping $F:\widetilde{W} \to \widetilde{W}$ is an invertible, bi-measurable and probability-preserving transformation and $T^+$ is non-singular.
Also, we take the $\sigma$-algebra $\mathcal{C}^+=\mathcal{B}_{\check{\Sigma}_+}\otimes \mathcal{B}_{\widetilde{W}}$ on $Y^+$.
Then, the system $(Y^+, \mathcal{C}^+, ((p_i^x),(\pi_{ij}^x)), T^+)$ is a random Markov chain.
\begin{proposition}\label{prott}
There exists a measurable map $x \mapsto (\mathbb{P}^x_{i})_{i \in\mathbb{N}}$ such that the measure $d\mathbb{P}^x d\mu_F$ on $(Y^+,\mathcal{C}^+)$ given by
\begin{equation*}
\mathbb{P}^x([\zeta^x]):=\mathbb{P}_{i_1}^x\prod_{k=1}^{n-1}\pi^{F^{k-1}(x)}_{i_k i_{k+1}}
\end{equation*}
is $T^+$-invariant, where $\zeta=(i_1,\ldots,i_n)$ is a finite word of alphabets $\mathbb{N}$ and $(\mathbb{P}^x_{i})_{i \in\mathbb{N}}$ is a probability vector, that is $\mathbb{P}^x_{i} \in [0,1]$ and $\sum_{i\in\mathbb{N}}\mathbb{P}^x_{i} =1$.
\end{proposition}
\begin{proof}
To prove, we verify conditions ($a$) and ($b$) of Proposition \ref{propo1} for $(Y^+, \mathcal{C}^+, ((p_i^x),(\pi_{ij}^x)), T^+)$ as defined above.
By (\ref{h1}) and (\ref{h11}), if for $x \in \widetilde{W}$ and $i, j, k\in\mathbb{N}$ one has $\pi^x_{ij}>0$, $\pi^x_{ik}>0$, $p_j^{F(x)}>0$ and $p_k^{F(x)}>0$, then $j=k$ and $\pi^x_{ij}=\pi^x_{ik}=p_j^{F(x)}=p_k^{F(x)}=1$.
Thus, if we take $C_x:=2$, for each $x \in \widetilde{W}$, we have
\begin{equation*}
  C_{F(x)}^{-1}< \frac{\pi^x_{ij}}{\pi^x_{ik}}.\frac{p_k^{F(x)}}{p_j^{F(x)}}< C_{F(x)}.
\end{equation*}
Thus condition ($a$) is satisfied. Moreover, if we take $\eta_x :=1/2$ then
\begin{equation*}
  \sum_{j: \pi^x_{ij}>0}p_j^{F(x)}> \eta_{F(x)}
\end{equation*}
which verifies condition ($b$).
Now, the proof can be followed by Proposition \ref{propo1}.
\end{proof}
Now, consider an extension $(Y, \mathcal{C}, T)$ of $(Y^+, \mathcal{C}^+, T^+)$ as follows. Take $Y:=\check{\Sigma} \times \widetilde{W}$ and define the extension map $T:Y \to Y$ by
$T(\textbf{i},x)=(\sigma_{\mathcal{A}}(\textbf{i}),F(x))$.
Consider the Borel $\sigma$-algebra $\mathcal{C}$ on $Y$.
By Kolmogorov's extension theorem, $(Y,\mathcal{C}, T)$ admits an invariant measure $\nu$
inherited from the invariant measure $d\mathbb{P}^x d\mu_F$ on $(Y^+,\mathcal{C}^+, T^+)$ given by the previous proposition.
Furthermore, by Proposition \ref{l}, the one-sided topological Markov shift $(\Sigma_{\mathcal{A}}^{+},\sigma_{\mathcal{A}}^+)$ is topologically mixing.
Also, by part (i) of Remark \ref{exact}, the induced map $F$ is topologically mixing.
These observations imply that the measure $d\mathbb{P}^x d\mu_F$ and hence the measure $\nu$ are ergodic.
Thus, we get the next result.
\begin{corollary}\label{remtt}
Consider the above extension $(Y, \mathcal{C}, T)$ of $(Y^+, \mathcal{C}^+, T^+)$. Then $(Y,\mathcal{C}, T)$ admits an ergodic invariant measure $\nu$
inherited from the invariant measure $d\mathbb{P}^x d\mu_F$ given by Proposition \ref{propo1}.
\end{corollary}
Take
\begin{equation*}
  Z:= Y \bigcap \big(\bigcup_{i\in\mathbb{N}}C(i)\times M_i \big)\quad \text{and} \quad H:=\rho\times id |_{Z},
\end{equation*}
where $\rho$ is the injective map defined by (\ref{rho}).
Then, $H \circ T|_{Z}=F_A^{\psi} \circ H$ which means that $H$ is a conjugacy between $T |_{Z}$ and $F_A^\psi $.
Put
\begin{equation}\label{m2}
m:=H_*\nu.
\end{equation}
The support of $m$ is contained in the subset $\mathbb{W}$ and it can be extended to a probability measure on the whole space $\Sigma_A\times M$ by taking $m(A)=0$ for each measurable set
$A \subset \Sigma_A\times M \setminus \mathbb{W}$. In particular, $m$ is an $F_A^\psi$-invariant probability measure. Moreover, by Corollary \ref{remtt}, it is also ergodic.
Hence, we get the next result.
\begin{proposition}\label{p66}
Measure $m$ given by (\ref{m2}) is an $F_A^\psi$-invariant ergodic probability measure. Furthermore, the support of $m$ is contained in $\mathbb{W}$.
\end{proposition}

The next result shows that the measure $m$ given by (\ref{m2}) can be lifted to an $F_A$-invariant measure.
As a consequence, the subset $ \mathcal{M}_L(F_A,\mathbb{Y})$ of liftable measures is nonempty. Before we recall the following proposition (Proposition 4.1 in \cite{PSZ2}).

\begin{proposition}\label{stationary}
Consider the measure $m$ given by (\ref{m2}). Then $Q_m < \infty$ and $\mathcal{L}(m) \in \mathcal{M}(F_A, \mathbb{Y})$.
\end{proposition}
\begin{proof}[proof of Proposition \ref{p66}]
Let us take the constant $0<\sigma<1$ given by condition (\textbf{C1}) of Definition \ref{expand}. Then, by definition of $\psi$, we get
$$Q_{m}=\int_{\mathbb{W}}\psi dm=\sum_{i\in\mathbb{N}}\tau(M_i^\circ)\mu_F(M_i^\circ)\leq \sum_{i\in\mathbb{N}}\tau(M_i^\circ)\sigma^{\tau(M_i^\circ)}<\infty.$$
By Proposition \ref{stationary}, this completes the proof.
\end{proof}
Theorem C is now an immediate corollary of Propositions \ref{p66} and \ref{stationary}.
\subsection{Thermodynamic formalism with respect to liftable measures}\label{section8}
Following the approach used in \cite{PSZ2}, for the induced skew product map $F_A$, we will establish a thermodynamic formalism for liftable measures.
In particular, we present some conditions on potentials which guarantee existence and uniqueness of the associated
equilibrium measures. These measures minimize the free energy among the liftable measures.
We start with some preparatory results.

Let $\phi: \Sigma_A\times M \to \mathbb{R}$ be a Borel function, call {\it potential}. We assume
that $\phi$ is well-defined and is finite at every point $(\textbf{i},x) \in \Sigma_A\times M$. For a potential $\phi$, a measure $\mu_\phi$, in the space $\mathcal{M}_L(F_A,\mathbb{Y})$ of liftable measures, will be called an \emph{equilibrium measure} for
$\phi$ if
\begin{equation*}
  P_L(\phi):=\sup_{\mu \in \mathcal{M}_L(F_A,\mathbb{Y})}\{h_\mu(F_A)+\int_{\Sigma_{A}\times M} \phi d\mu\}=h_{\mu_\phi}(F_A)+\int_{\Sigma_{A}\times M} \phi d\mu_\phi.
\end{equation*}
This definition of equilibrium measures only considers liftable measures and it differs from the standard one. We define the induced potential $\overline{\phi}: \mathbb{W} \to \mathbb{R}$ by
\begin{equation}\label{pot}
\overline{\phi}(\textbf{i},x) :=\sum_{\ell=0}^{\psi(\rho(C(i))\times M_i)-1}\phi (F_A^\ell(\textbf{i},x)), \ (\textbf{i},x)\in \rho(C(i))\times M_i.
\end{equation}

By using the next result, we provide a relation between the thermodynamic formalism of original and induced systems.
It is a reformulation of Abramov's and Kac's formula in our setting.
\begin{proposition}\label{kac}
Let $m \in \mathcal{M}(\mathbb{F}, \mathbb{W})$ with $Q_m < \infty$, where $Q_m$ is given by (\ref{q1}).
Then
\begin{equation*}
  h_m(\mathbb{F})=Q_m .h_{\mathcal{L}(m)}(F_A)<\infty.
\end{equation*}
Let $\phi$ be a potential and $\overline{\phi}$ the corresponding induced potential. If $\int_{\mathbb{W}} \overline{\phi} dm$ is finite then
\begin{equation*}
  -\infty <\int_{\mathbb{W}} \overline{\phi} dm=Q_m \int_{\mathbb{Y}} \phi d\mathcal{L}(m) < \infty.
\end{equation*}
\end{proposition}
\begin{proof}
For the proof of Abramov's formula we refer to \cite{Z} (note that we require
the topological entropy of $F_{A}$ to be finite). To prove Kac's formula, using the definition of $\mathcal{L} (m)$, we get
\begin{align*}
 \int_{\mathbb{W}} \overline{\phi} dm
  =& \int_{\mathbb{W}} \sum_{\ell=0}^{\psi(\textbf{i},x)-1}\phi (F_A^\ell(\textbf{i},x)) dm(\textbf{i},x) \\
  =& \sum_{\rho(C(i))\times M_i} \sum_{\ell=0}^{\psi(\rho(C(i))\times M_i)-1} \int_{\rho(C(i))\times M_i} \phi (F_A^\ell(\textbf{i},x)) dm_{|(\rho(C(i))\times M_i)}(\textbf{i},x) \\
 =& \sum_{\rho(C(i))\times M_i} \sum_{\ell=0}^{\psi(\rho(C(i))\times M_i)-1} \int_{\mathbb{Y}} \phi(\textbf{j},y)dm(F_A^{-\ell}(\textbf{j},y)\cap \rho(C(i))\times M_i )\\
 =& Q_m \int_{\mathbb{Y}} \phi d\mathcal{L}(m).
\end{align*}
\end{proof}
In the previous subsection, we obtained an inducing scheme for $F_A$. This causes that $F_A$
can be modeled by towers over the subshift on the space of two-sided infinite sequences with a countable alphabet.
Thermodynamic formalism of this kind of subshifts studied in Section \ref{section6}.
In the rest, we use the results of Section \ref{section6} to establish a thermodynamic formalism for $F_A$ over the liftable measures.

For the induced potential $\overline{\phi}$ given by (\ref{pot}), we take
\begin{equation}\label{bar}
 \Phi:\Sigma_\mathcal{A} \to \mathbb{R}, \  \Phi(\textbf{i}):=\overline{\phi}(\rho(\textbf{i}), \pi(\textbf{i})),
\end{equation}
where $\pi$ is the coding map given by (\ref{code}).
Clearly, $ \Phi$ is well-defined on $\Sigma_\mathcal{A}$.
We call a measure $\nu_{\overline{\phi}}$ on $\mathbb{W}$ a Gibbs measure for $\overline{\phi}$ if the measure $(\rho \times \pi)^{-1}_\ast \nu_{\overline{\phi}}$
is a Gibbs measure for the function $\Phi$.
We call $\nu_{\overline{\phi}}$ an equilibrium measure for $\overline{\phi}$ if $- \int_{\mathbb{W}}\overline{\phi} d\nu_{\overline{\phi}}< \infty$ and
\begin{equation*}\label{gib}
  h_{\nu_{\overline{\phi}}}(\mathbb{F})+\int_{\mathbb{W}} \overline{\phi} d\nu_{\overline{\phi}} =\sup_{\nu \in \mathcal{M}(\mathbb{F},\mathbb{W}):  \ - \int_{\mathbb{W}}\overline{\phi} d\nu< \infty}\{h_\nu(\mathbb{F})+\int_{\mathbb{W}} \overline{\phi } d\nu \}.
\end{equation*}
Following \cite{PSZ2}, the potential $\overline{\phi}$
\begin{enumerate}
  \item [$(a)$] has \emph{summable variations} if the function $\Phi$ has summable variations, i.e.,
  \begin{equation*}
    \sum_{n \geq 1} V_n(\overline{\phi} \circ (\rho \times \pi))=\sum_{n \geq 1}V_n(\Phi)< \infty;
  \end{equation*}
  \item [$(b)$] has \emph{finite Gurevich pressure} if $P_G(\overline{\phi} \circ (\rho \times \pi))=P_G(\Phi)< \infty$.
 \end{enumerate}
\begin{proposition}\label{p1}
Assume that the function $\overline{\phi}$ has summable variations and finite
Gurevich pressure. Then
\begin{equation*}
  -\infty < P_L(\phi) < \infty.
\end{equation*}
\end{proposition}
\begin{proof}
By Remark \ref{rem222}, there is a periodic orbit for $\mathbb{F}$ in the set $\mathbb{W}$.
For the Dirac measure on that orbit, we have $\int_{\mathbb{Y}} \phi d\mu > -\infty$.
Then, since $h_\mu(F_{A})\geq 0$, we conclude that $P_L(\phi)> -\infty$. On the other hand, For every $\mu \in \mathcal{M}_L(F_{A},\mathbb{Y} )$ there exists a measure $m \in \mathcal{M}(\mathbb{F},\mathbb{W})$ with $\mathcal{L}(m)=\mu$, $Q_m < \infty $, and by
Proposition \ref{kac},
\begin{equation*}
  0 \leq h_m(\mathbb{F})=Q_m .h_{\mathcal{L}(m)}(F_{A})<\infty.
\end{equation*}
Take $\mu \in \mathcal{M}_L(F_{A},\mathbb{Y} )$ such that $\int_{\mathbb{W}} \overline{\phi}d m >-\infty$, where $\mathcal{L}(m)=\mu$.
Since $\overline{\phi}$ has summable variations and finite Gurevich pressure, one can show that it is bounded from above.
Hence $-\infty <\int_{\mathbb{W}} \overline{\phi}d m < \infty$, and by Proposition \ref{kac},
\begin{equation*}
 -\infty <\int_{\mathbb{W}} \overline{\phi}d m=Q_m \int_{\mathbb{Y}} \phi d\mu <\infty.
\end{equation*}
If $P_L(\phi)$ is nonpositive the upper bound is immediate. If $P_L(\phi)$ is positive, using the fact that $1 \leq Q_m < \infty$,
we closely follow \cite[Theorem~4.2]{PS} to get $P_L(\phi)<\infty$ and
omit the details.
\end{proof}
\begin{remark}\label{conj}
Notice that the existence and uniqueness of equilibrium measures for the induced system $(\mathbb{F},\mathbb{W})$ is obtained by conjugating the induced system to the restriction of two-sided subshift $(\Sigma_{\mathcal{A}},\sigma_{\mathcal{A}})$ to the subset $\check{\Sigma}\subset \Sigma_\mathcal{A}$ defined by (\ref{hat})
by the conjugacy map $G:=\rho^{-1}\circ \pi_1$, where
$\pi_1: \Sigma \times M \to \Sigma$ is the natural projection to the first factor and $\Sigma=\rho(\check{\Sigma})$.
\end{remark}
Let us take
\begin{equation*}
  \phi^+:=\overline{\phi - P_L(\phi)}=\overline{\phi}-P_L(\phi)\psi
\end{equation*}
and let $\Phi^+:= \phi^+ \circ (\rho \times \pi)$ with $\rho \times \pi(\textbf{i})=(\rho(\textbf{i}),\pi(\textbf{i}))$, $\textbf{i}\in \Sigma_{\mathcal{A}}$, where $\pi$ is the coding map given by (\ref{code}). Also, we take
\begin{equation}\label{ex}
 \widetilde{\mathbb{W}}:=\{\widetilde{X}=(X_n)_{n\leq 0}:\mathbb{F}(X_n)=X_{n+1}\},
\end{equation}
where $X_n=(\textbf{i}_n,x_n)\in \mathbb{W}$ and define the extension map $\widetilde{\mathbb{F}}$ by $(\widetilde{\mathbb{F}}(\widetilde{X}))_n=X_{n+1}$.

The proof of the following has the same flavor as that in Section 4 of \cite{PSZ2}.
\begin{theorem}\label{maint}
Assume that $\Phi$ has strongly summable variations and $P_G(\Phi)<\infty$.
Assume also that $P_G(\Phi^+)<\infty$ and $\sup_{\textbf{i} \in \Sigma_\mathcal{A}}\Phi^+(\textbf{i})<\infty$. Then
\begin{enumerate}
  \item [$(a)$] there exists a $\sigma_\mathcal{A}$-invariant ergodic Gibbs measure $\nu_{\Phi^+}$ for $\Phi^+$;
  \item [$(b)$] if $h_{\nu_{\Phi^+}} (\sigma_\mathcal{A}) < \infty$, then $\nu_{\Phi^+}$ is the unique equilibrium measure for $\Phi^+$;
  \item [$(c)$] if $h_{\nu_{\Phi^+}} (\sigma_\mathcal{A}) < \infty$ then the measure $\nu_{\phi^+}:= (\rho \times \pi)_\ast \nu_{\Phi^+}$ is a unique $\mathbb{F}$-invariant
ergodic equilibrium measure for $\phi^+$;
  \item [$(d)$] if $P_G(\Phi^+)=0$ and $Q_{\nu_{\Phi^+}}<\infty$, then $\mu_\phi = \mathcal{L}(\nu_{\Phi^+} )$ is the unique ergodic
equilibrium measure in the class $\mathcal{M}_L(F_{A},\mathbb{Y} )$ of liftable measures.
\end{enumerate}
\end{theorem}
\begin{proof}
By Proposition \ref{p1}, $P_L(\phi)$ is finite and that $V_n(\Phi) = V_n(\Phi^+)$. By these facts and Theorem \ref{mt}, statements (a) and (b) can be followed.

To prove statement (c), consider the function $\varphi : \widetilde{\mathbb{W}} \to \mathbb{R}$
defined by $\varphi((X_n)_{n\leq 0}):=\overline{\phi}(X_0)$, and let $\Phi=\varphi \circ \pi^+ \circ G^{-1}$ and $\Phi^+=\varphi^+ \circ \pi^+ \circ G^{-1}$,
where $G$ is the conjugacy map given by Remark \ref{conj}, $\varphi^+ = \varphi- P_L(\phi)\psi$, $\pi^+$
is the projection defined by $\pi^+((X_n)_{n \in \mathbb{Z}})=(X_n)_{n \leq 0}$ for $X_n \in \mathbb{W}$ with $\mathbb{F}(X_n)=X_{n+1}$ and $\psi$ is the inducing time given by (\ref{psi}).

By Proposition \ref{proin}, the set $\check{\Sigma}$ has full measure by $\nu_{\Phi^+}$.
Thus, by Proposition \ref{pi}, the measure $\widetilde{\nu}_{\varphi^+} := (\pi^+ \circ G^{-1})_\ast \nu_{\Phi^+}$
is the unique $\widetilde{\mathbb{F}}$-invariant ergodic equilibrium measure for $\varphi^+$, i.e.,
\begin{equation*}
  h_{\widetilde{\nu}_{\varphi^+}}(\widetilde{\mathbb{F}})+ \int \varphi^+ d\widetilde{\nu}_{\varphi^+}=\sup_{\nu \in \mathcal{M}(\widetilde{\mathbb{F}},\widetilde{\mathbb{W}})}\{h_\nu(\widetilde{\mathbb{F}})+\int \varphi^+ d\nu\},
\end{equation*}
where $\mathcal{M}(\widetilde{\mathbb{F}},\widetilde{\mathbb{W}})$
denotes the set of ergodic $\widetilde{\mathbb{F}}$-invariant Borel probabilities on $\widetilde{\mathbb{W}}$.
It is known that for given $\nu \in \mathcal{M}(\mathbb{F},\mathbb{W})$ and $\widetilde{\nu} \in \mathcal{M}(\widetilde{\mathbb{F}},\widetilde{\mathbb{W}})$ the lift of $\nu$,
$h_\nu(\mathbb{F})=h_{\widetilde{\nu}}(\widetilde{\mathbb{F}})$.
Also, for any $\overline{\phi}\in L^1(\nu)$ the function $\varphi$ on $\widetilde{\mathbb{W}}$
satisfies $\int \overline{\phi}d \nu=\int \varphi d\widetilde{\nu}$. Since $\nu_{\phi^+}$ is the image of $\widetilde{\nu}_{\varphi^+}$ under this
correspondence, (c) is verified.

Also, $Q_{\nu_{\phi^+}}<\infty$ and Proposition \ref{p1} imply $h_{\nu_{\Phi^+}}(\sigma_{\mathcal{A}})<\infty$. Hence, (c) is satisfied. Finally, statement (c) and Proposition \ref{p1} imply
\begin{equation*}
  0=h_{\phi^+}(\mathbb{F})+\int \phi^+ d\nu_{\phi^+}=\frac{1}{Q_{\nu_{\phi^+}}}(h_{\mu_{\phi}}(F_{A})+\int \phi d\mu_\phi -P_L(\phi))
\end{equation*}
hence
\begin{equation*}
h_{\mu_{\phi}}(F_{A})+\int \phi d\mu_\phi=P_L(\phi).
\end{equation*}
\end{proof}
The following theorem is a counterpart of \cite[Theorem~4.6]{PSZ2} to our setting.
\begin{theorem}\label{t55}
Let $(\mathcal{S}, \psi)$ be an inducing scheme for the induced skew product $F_{A}$ corresponding to
the inducing scheme of hyperbolic type $(\mathcal{M}, \tau)$ for a topologically mixing locally conformal
expanding semigroup $\Gamma$. Let $\phi \colon \Sigma_{A} \times M \to \mathbb{R}$ be a potential and $\Phi$ is given by (\ref{bar}).
 The following statements hold:
\begin{enumerate}
  \item [$(1)$] Condition $(P1)$ implies that the function $\Phi$ has strongly summable variations.
  \item [$(2)$] Condition $(P2)$ implies the finiteness of the Gurevich pressure $P_G(\Phi) < \infty$
and \\$\sup_{\textbf{i} \in \Sigma_{\mathcal{A}} }\Phi(\textbf{i}) < \infty$.
  \item [$(3)$] Condition $(P3)$ implies that $\sup_{\textbf{i} \in \Sigma_{\mathcal{A}} }\Phi^+(\textbf{i}) < \infty$ and $P_G(\Phi^+) = 0$.
  \item [$(4)$] If $\Phi$ has strongly summable variations, then Condition $(P2)$ implies $Q_{\nu_{\Phi^+}}<\infty$.
\end{enumerate}
\end{theorem}
Theorem D is now an immediate corollary of Theorems \ref{maint} and \ref{t55}.

\end{document}